\pdfoutput=1
\documentclass[12pt]{amsart}
\usepackage[margin=1in]{geometry}

\usepackage[final]{microtype}
\usepackage[x11colors]{xcolor}
\usepackage{tikz}
\usetikzlibrary{shapes,decorations.pathreplacing,decorations.pathmorphing,arrows,positioning,matrix}
\tikzset{>=stealth}
\usepackage{amsmath,amssymb,amsthm,mathtools}
\usepackage[scr=boondox]{mathalfa}
\usepackage[colorlinks]{hyperref}
\usepackage{listings}

\numberwithin{equation}{section}

\newtheorem{thm}{Theorem}
\numberwithin{thm}{section}
\newtheorem{prob}[thm]{Problem}
\newtheorem{lemma}[thm]{Lemma}
\newtheorem{cor}[thm]{Corollary}

\newtheorem{defn}[thm]{Definition}
\theoremstyle{remark}
\newtheorem{rmk}[thm]{Remark}
\newtheorem{ex}[thm]{Example}

\newcommand{\defeq}{\ensuremath\coloneqq}
\DeclareMathOperator{\len}{len}

\DeclareMathOperator{\Val}{\mathbb{V}}
\DeclareMathOperator{\Msr}{\mathbf{M}}

\DeclareMathOperator{\red}{\downarrow}
\DeclareMathOperator{\tot}{\red_{[]}}
\DeclareMathOperator{\Merge}{Merge}
\DeclareMathOperator{\JOINS}{Joins}
\newcommand{\pair}[1]{\ensuremath\left\langle#1\right\rangle}

\newcommand{\sS}[1]{\ensuremath\mathcal{#1}}

\newcommand{\Zb}{\ensuremath\mathbb{Z}/2\mathbb{Z}}
\newcommand{\into}{\ensuremath\hookrightarrow}
\newcommand{\onto}{\ensuremath\twoheadrightarrow}
\newcommand{\ptimes}{\ensuremath\tilde{\times}}

\begin{document}
\title[Obstructions to Data Merging]{Persistent Obstruction Theory for a Model Category of Measures with Applications to Data Merging}
\author{Abraham D. Smith, Paul Bendich, and John Harer}
\address{A. Smith\\
Mathematics, Statistics, and Computer Science\\
University of Wisconsin-Stout\\
Menomonie, WI, 54751, USA\\
Geometric Data Analytics, Inc. \\
Durham, NC, 27707, USA}
\email{smithabr@uwstout.edu}
\address{P. Bendich \\
Mathematics \\
Duke University \\
Durham, NC, 27708, USA \\
Geometric Data Analytics, Inc. \\
Durham, NC, 27707, USA \\
United States}
\email{bendich@math.duke.edu}
\address{J. Harer \\
Mathematics\\
Duke University\\
Durham, NC, 27707, USA \\
Geometric Data Analytics, Inc. \\
Durham, NC, 27707, USA}
\email{harer@math.duke.edu}
\date{\today}


\begin{abstract}
Collections of measures on compact metric spaces form a model category (``data
complexes''), whose morphisms are marginalization integrals.  The fibrant
objects in this category represent collections of measures in which there is a
measure on a product space that marginalizes to any measures on pairs of its
factors.  The homotopy and homology for this category allow measurement of
obstructions to finding measures on larger and larger product spaces.
The obstruction theory is compatible with a fibrant filtration built from the
Wasserstein distance on measures.

Despite the abstract tools, this is motivated by a widespread problem in data science.
Data complexes provide a mathematical foundation for semi-automated data-alignment tools that 
are common in commercial database software.
Practically speaking, the theory shows that database JOIN
operations are subject to genuine topological obstructions.  Those obstructions
can be detected by an obstruction cocycle and can be resolved by moving through
a filtration.  Thus, any collection of databases has a persistence level, which measures
the difficulty of JOINing those databases.
Because of its general formulation, this persistent obstruction theory also
encompasses multi-modal data fusion problems, some forms of Bayesian inference,
and probability couplings.
\end{abstract}

\maketitle

\section{Introduction}
\label{sec:Intro}

We begin this paper with an abstraction of a problem familiar to any large enterprise.
Imagine that the branch offices within the enterprise have access to many data sources.  
The data sources exposed to each office are related and overlapping but non-identical.
Each office attempts to merge its own data sources into a cohesive
whole, and reports its findings to the head office.  This is done by humans, often aided by ad-hoc data-merging software solutions. Presumably, each office does a good job of merging the data
that they see.
Now, the head office receives these cohesive reports, and must combine them
into an overall understanding.

This paper provides a mathematical foundation combining methods from measure theory, simplicial homotopy, obstruction theory, and persistent cohomology (Section \ref{sec:IO} gives an overview) for semi-automated data-table-alignment tools (e.g, HumMer \cite{FelixNaumann}) that
are common in commercial database software. 
Data tables are abstracted as measures over value spaces.
The problem of merging tables, or indeed merging previously-merged tables, is recast as the search for a measure
that marginalizes correctly. 

For example, one data table might record the ages and heights and weights of patients in a hospital, abstracted as a formal sum of point-mass atoms.
Another data table might be an actuarial table giving ages and heights and weights for an entire population, abstracted as a smooth distribution where the heights and weights form 2-dimensional elliptical Gaussian distributions for each age and height, the means and singular values varying with age.
Both examples would be data tables on the same age-by-height-by-weight space.
A third data table might be a simulated probabilistic model of injury severity during vehicle collisions based on height and weight of the passenger.
This data table on height-by-weight-by-severity space may or may not merge with each of the previous examples over height-by-weight, within some error.
One can imagine many other data tables collected from myriad sources (motor-vehicle records, longitudinal studies, clinical trials) related to this example that may be of interest.

Our first fundamental result (Theorem \ref{thm:ft}) uses this measure-theoretic lens to draw a surprising correspondence between the process of JOIN in database engineering and the Kan extension property for simplicial sets.

This abstraction, and the model-theoretic tools that come with it, permits several advances over the current state of the art, which are collected in our second fundamental result (Theorem \ref{thm:tri}).
First, inconsistencies in table collections are automatically detected as \emph{obstructions} in the sense of Steenrod (i.e, a certain co-cycle is not zero).
Second, when inconsistencies are detected, the obstruction theoretic tools, combined with persistent cohomology, provide two potential remedies: a) if algebraically permitted (i.e, a certain co-class is zero), the user may retreat back one level of merging, repair, and proceed; b) else, the user may settle for a measure that only marginalizes approximately correctly, with the degree of possible correctness computed automatically by persistent obstruction theory.

More broadly, we are interested in the following three meta-problems:
\begin{prob}[Testing Problem]
Given several sources of partial information, how do we
test that a hypothetical explanation is reasonably consistent with that
partial information?
\label{GPtest}
\end{prob}

\begin{prob}[Merging Problem]
Given several sources of partial information, how do we merge that
partial information into a cohesive whole?
\label{GPmerge}
\end{prob}

\begin{prob}[Meta-Merging Problem]
Given many sources of partial information, and several partial attempts to merge
some combinations of them, is there a way to merge these partial merges into
a cohesive whole?
\label{GPmetamerge}
\end{prob}
By ``sources of partial information'' we mean, roughly, collected data (databases, spreadsheets, pictures), statistical models, established theories,
simulations, and general population trends.
In this article, we define a formal mathematical structure---a \emph{Data
Complex} (Section \ref{sec:db})---that can encapsulate a wide range of problems like \ref{GPtest},  \ref{GPmerge},
and \ref{GPmetamerge}.
A data complex is meant to encode each source of information as a
finite measure on an appropriate value space.  Whether these measures arise from collected data
as in Problem \ref{GPmerge}
or some model/theory/simulation/trend/merger derived from previous work as in
Problem \ref{GPmetamerge}, we call them \emph{data tables}.
By using measures, we combine Problems \ref{GPmerge} and \ref{GPmetamerge} into a
single problem.

\subsection{Overview of Technical Approach}
\label{sec:IO}

Often, formal mathematics looks very different than its intuitive purpose, so
we want to make sure the reader understands our intent, absent its formal
notation.

We want a mathematically rigorous way to encode and solve Problems~\ref{GPtest}--\ref{GPmetamerge}. 
When translated to the language of data complexes, a physically reasonable process for
``merge [data tables] into a cohesive whole'' can be expressed in terms of four
key mathematical ingredients: homological algebra for simplicial sets, simplicial homotopy
\cite{Kan1958, Quillen1967}, obstruction theory
\cite{Steenrod}, and persistent (co)homology across a filtration \cite{Edelsbrunner2010}.

The first ingredient (homological algebra) is used because 
data tables may overlap partially, meaning that we need a formal simplicial
complex to encode all possible intersections of all possible data tables.  
Moreover, simplicial sets allow data tables with repeated columns.
The
marginalization integral provides a face map and a boundary operator,
and an analogue of the diagonal measure within a product provides the degeneracy map.
The face and boundary operators tell us whether one ``narrow'' data table
reproduces the data of another ``wider'' data table via marginalization
integrals.  Thus, the question of whether several ``narrow'' data tables can be merged
into a few ``wider'' data table becomes the question of whether a $k$-chain is the
boundary of a $(k{+}1)$-chain.  That is, the ability to merge overlapping
partial information sources as in Problem \ref{GPmerge} is encoded as the homology of a data complex.

The second ingredient (simplicial homotopy)
arises because Problems \ref{GPtest}, \ref{GPmerge},  \ref{GPmetamerge} suggest that we want ``simple''
solutions.  Even when partial merging is possible in the sense of homology, it
may be that the result is too complicated to be merged further.  In the study
of geometric bundles, the fundamental group and higher homotopy groups of the
fiber play a key role, and we use simplicial homotopy in a similar way here. A
simple solution to Problem \ref{GPmerge}/\ref{GPmetamerge} or a simple hypothesis in Problem \ref{GPtest} corresponds to
a single data table (as opposed to a complicated chain of many data tables),
which is indicated by triviality in the simplicial homotopy group. 

An
important side effect of introducing simplicial homotopy (via model categories) is that we see that
the Kan extension condition means ``merging operations are possible.'' The
process we call \emph{merging} is similar to \emph{JOIN'ing} in database
software, to \emph{fusion} in multi-modal data analysis, and to \emph{coupling}
in probability theory.  This link reinforces the intuition that data complexes
are a good way to encode Problems \ref{GPtest}/\ref{GPmerge}/\ref{GPmetamerge} for modern data mining
when using spreadsheets, DataFrames, and SQL databases. 
Indeed, our first fundamental result (Theorem \ref{thm:ft}) explicitly formalizes this correspondence.

The reader may be wondering why we introduce something as abstract as
simplicial homotopy into something so concrete and common as data merging.
Consider the typical database operation 
\begin{lstlisting}[language=SQL,frame=single]
SELECT * FROM table1 INNER JOIN table2
ON table1.column1 = table2.column2
WHERE condition;
\end{lstlisting}
When issuing such a command, the administrator \emph{must} designate two tables
to be \texttt{JOIN}ed and choose specific columns from the two tables to be identified
via the \texttt{ON} clause.   The \texttt{SELECT * \dots;} command returns a table, whose columns must
appear in some order that is determined by the ordering of attributes in \texttt{table1} and
\texttt{table2}, by their placement in the command, and by the columns in the
\texttt{ON} clause.   Thus, in the language of
Section~\ref{sec:db}, the database software and the working administrator must 
agree on a total set of attributes, the attributes in each table, and an
ordered attribute inclusion to be used for the \texttt{ON} clause.

This command also indicates why we formalize ``data tables'' as measures over
products of attribute value spaces.  Replacing \texttt{SELECT *} with a
\texttt{SELECT columnList} corresponds to the ability to
re-order the attribute list and to marginalize the output to a sublist of attributes; hence, arbitrary finite products are possible.
The optional \texttt{WHERE condition} clause allows one to impose additional restrictions on the values to
be considered by imposing logical conditions on the entries, such as \texttt{WHERE (age > 18 AND height > 200)}.  These conditions allow one one to restrict the data table to any\footnote{Any measurable subset---in principle and given a sufficiently generous SQL implementation.} measurable subset of the value space. 
The entries of a \texttt{WHERE}-restricted data table constitute the mass of this measurable set, with respect to the data table.
(Finally, for those fluent in SQL subtleties, note that the ability to perform
\texttt{LEFT}, \texttt{RIGHT}, and \texttt{OUTER JOIN} 
instead of \texttt{INNER JOIN} will be encompassed by approximate join and face operations
in Section~\ref{sec:FO}.)

The third ingredient (Steenrod's obstruction theory as in \cite{Steenrod1943}) provides guidance
on how to combine homological algebra and homotopy theory to detect and
describe any obstructions to an iterative merging process.
In its original formulation, obstruction theory asks 
whether a section $\sigma$  of a fiber bundle $p$ defined
over a topological space $B$ can be extended to a section defined over a
larger topological space $A \supset B$? 
The most famous example is the smooth category,
where one computes characteristic cohomology classes to indicate whether sections of a bundle can be extended globally.
Steenrod studied this problem in the case of fibrations over general
topological spaces.  Typically, assuming one has some sort of CW structure on $A$, one tries to extend $\sigma$ first over the $0$-skeleton of $A$ and then the $1$-skeleton of $A$, and so forth. Assuming one has already extended $\sigma$ to a section over the $(n-1)$-skeleton of $A$, Steenrod's obstruction cocycle is an element $\xi_\sigma$ of $C^{n}(A,B;
\pi_{n-1, \sigma}(F_0))$, where $\pi_{n-1, \sigma}(F_0)$ is the homotopy group of the fiber $F_0$ of $p$, as twisted by $\sigma$; loosely, the co-chain $\xi_\sigma$ is defined on each $n$-cell $c$ of $A$ by restricting $\sigma$ to the boundary of $c$, but there is some nuance coming from the twisting needed to turn this into a homotopy class of the fiber.
If this cocyle $\xi_\sigma:C_n \to \pi_{n-1}$ is trivial in
homotopy, then the section $\sigma$ can be extended, and otherwise it cannot. However, if this cocycle is a coboundary,
$[\xi_\sigma]=0$, then there is another section $\tau$ , agreeing with $\sigma$ on the $(n-2)$-skeleton of $A$, with 
$\xi_{\tau}$ trivial in homotopy.  Hence, obstructions are
discovered dimension-by-dimension via homotopy-valued cohomology, and the obstruction computation can often permit the ``correction'' of initial extensions of the section to avoid higher-dimensional obstructions.

This concept of an obstruction cocycle was introduced by Steenrod in \cite{Steenrod1943} and
revisited many times, such as \cite{Olum1950} and \cite{Lundell1960}.
Its importance motivated early work in category theory.
The entire \emph{raison d'\^etre} of defining fibrant objects and model
categories in \cite{Kan1958} and \cite{Quillen1967} was to establish the most
general context in which these (co)homology and homotopy calculations remain
sensible for more general notions of ``weak equivalence.''
In particular, one does not require actual topological spaces to perform obstruction
theory, merely fibrant objects in a model category. 

Here, we establish homology and homotopy theory for data tables by
relying on these categorical foundations, giving us an obstruction theory directly
analogous to Steenrod's.
When sequential merging is impossible, the obstruction
cochain can compute specific data tables that obstruct the process.  That is,
obstruction theory determines when Problem \ref{GPmetamerge} is solvable locally but
not globally.  

The fourth ingredient, persistent (co)homology, provides a mathematically
robust way to measure how much the underlying original data tables would have
to be altered, in order to overcome an obstruction.  This is a key feature of
the theory, because from a practical perspective, multiple information sources
are \emph{never} perfectly consistent.  Typos and transpositions and omissions
and error bars always exist, and must be accounted for.  We use a filtration built
from the Wasserstein distance on measures to ensure that the desired simplicial homotopy
is possible throughout all levels of the filtration.
This allows for a well-defined notion of \emph{persistent obstruction theory.}
Our second fundamental result (Theorem \ref{thm:tri}) formalizes the idea that when inconsistencies are identified, one of two remedies may be available\footnote{In fact, the second is always available, but may be less desirable!}
\begin{itemize}
 \item the head office should retreat back one merging level, repair (with repair suggested by algebra), then again seek consensus
 \item the head office should settle for only approximate consensus, where the desired measure only approximately marginalizes correctly, with the degree of approximation computed via persistence.
\end{itemize}

The ultimate result of this article is a mathematically robust framework for
data merging that is reasonably applicable to real-world
data.  In this framework, Problems~\ref{GPtest} and
\ref{GPmerge}/\ref{GPmetamerge} become Problems~\ref{GPtest2} and \ref{GPmerge2}, which are answered by Theorem~\ref{thm:tri} and Definition~\ref{def:pers}.

\subsection{Related Work}
\label{sec:RW}

To the best of our knowledge, this is the first work to combine all of the tools above to build a robust obstruction theory
for databases. Other authors have used different aspects of these tools to address databases.
For example, recent work by Fong and Spivak uses database schemas and type/value
relationships as a motivational example to introduce functors and (co)limits
\cite[Chapter 3]{Fong2018}.  Specifically Example 3.99 and the chapter's final
remark are somewhat in the same spirit as the approach taken here.
Our category of data complexes in Section~\ref{sec:db} is similar to the categorical presentation in
\cite{Schultz2015, Wisnesky2015, Schultz2017}, but our data tables are built from measures (not
sets) in order to flexibly address the errors that are inevitable in applications.
Finally, other recent work by Abramsky, Morton, and collaborators uses obstruction
theory (in the sheaf-theoretic context) to detect non-contextuality in quantum
theory, with an application to  the non-existence of a universal data table
that contains a set of given tables
\cite{Abramsky2013,Abramsky2015,Morton2017}.
We expect that further interweaving of these measure-theoretic,
sheaf-theoretic, and simplicial/categorical perspectives will be fruitful in the future.

\subsection{Outline}
\label{sec:Out}

The rest of this paper is organized as follows.
Section~\ref{sec:db} defines the basic object of study, a \emph{data complex}, and draws a mapping between its simplicial set structure and the choices that must be 
made by any database administrator.
Categorical language is alluded to in this section, but a full categorical treatment of data complexes is confined to the Appendix.
Section~\ref{sec:homotopy} connects simplicial homotopy to the notion of JOIN, and shows how obstruction theory detects the impossibility of merging.
Section~\ref{sec:FO} describes our notion of persistent obstruction theory and its application to the idea of fuzziness of consensus.
The paper concludes with discussion of practical considerations for
applications in Section~\ref{sec:Disc}.

\section*{Acknowledgments}

Work by all three authors was partially supported by the DARPA Simplex Program, under contract \# HR001118C0070.
The last two authors were also partially supported by the Air Force Office of Scientific Research under grant AFOSR FA9550-18-1-0266.

We are very grateful to John Paschewitz and Tony Falcone for project guidance and technical direction, and to
Greg Friedman, Justin Curry, Jose Perea, and Jonathan Mattingly for helpful discussions at various stages of the theoretical development.

\section{Attributes and Data Tables}
\label{sec:db}

This section provides a
practical developmental discussion of a Data Subcomplex that should be accessible
to a fairly wide mathematical audience, with full categorical language found in the Appendix. The basic definitions appear in Sections \ref{sec:DS} and \ref{sec:operations}, culminating in Theorem \ref{thm:sset} which shows that we have indeed defined
a simplicial set. Operations that are specifically useful to standard database operations (inclusion/merge/join) are defined in Section \ref{sec:IMJ}.
Then Section \ref{sec:sec} makes plain the analogue of ``section of a bundle,'' which permits the rephrasing of our fundamental problems in mathematical language, and Section \ref{sec:chains} defines the (co)homology of data complexes needed for obstruction theory.

\subsection{Data Subcomplex as a Simplicial Set}
\label{sec:DS}
Our definitions are aimed at making precise the following real-life scenario in data
administration.
\begin{enumerate}
\item The administrator chooses a set $A$ of all attributes (column names and variable
types) of interest.
\item For each attribute $a$ in the list $A$, the administrator chooses a space
of possible values, and a ``reasonable'' metric $\rho_a$ that can provide the
distance between any two values in that space.  Our notion of ``reasonable''
includes compactness, which is typically guaranteed by boundedness of
realistic integer or vector-valued entries.
\item The administrator acquires ``data'' for some
lists of attributes, and attempts to reconcile these into a joint view
across all attributes in $A$.   The reconciliation process involves ``join''
operations that could be represented by SQL commands such as
\begin{lstlisting}[language=SQL,frame=single]
SELECT * FROM table1 INNER JOIN table2
ON table1.column1 = table2.column2
WHERE condition;
\end{lstlisting}
\item When reconciling, the administrator may choose
to alter the data, as long as the alterations are ``small'' with respect to
both the individual values via $\rho_a$ and with respect to the overall
information-theoretic content of the data. 
\end{enumerate}

The former two items are choices that must be made.  The latter two items are a
process to be accomplished.  The mathematical structure developed here 
is informed deeply by the example SQL command, as discussed in
Section~\ref{sec:IO}.

Let us define our objects.  It is convenient to use language of category
theory; see Appendix~\ref{sec:sset} for our conventions.

Consider a finite set $A$.  The elements are called
\emph{attributes}.  For each attribute $a \in A$, there is a
compact metric space $(\Val(a), \rho_a)$, called the \emph{value
space}.\footnote{These assumptions imply that $\Val(a)$ is complete, separable and is a Radon space.  In
many applications, the space $\Val(a)$ is finite or a closed interval in $\mathbb{R}$, so one needn't imagine
esoteric spaces to grasp the theory.}
These assumed objects (the finite set of attributes and
a compact metric space assigned to each attribute) are user-supplied by
a data administrator; after these choices are made, everything else proceeds as defined.

Each $\Val(a)$ is a Radon space (in particular, a measurable space) using the usual Borel algebra from the
metric $\rho_a$.
These metrics will be used in Section \ref{sec:FO} to quantify levels
of acceptable imprecision when marginalizing measures.

An \emph{attribute list} $T = [a_0, a_1, \ldots, a_n]$ is a finite sequence of
attributes; that is, an attribute list is a function $T:\{0,\ldots, n\} \to A$.
The \emph{length} of an attribute list is $\len(T) \defeq n+1$.
An attribute list $T$ is called \emph{nondegenerate} if it contains no
repetitions; that is if the function $T$ is one-to-one.
The longest nondegenerate attribute lists are permutations of $A$.

For any attribute list $T$, the product space $\Val(T) \defeq \prod_{i=0}^n \Val(a_i)$
is well-defined.  The product space $\Val(T)$ 
admits the $L^\infty$ metric $\rho_T = \max_{a \in T} \rho_a$ 
and is measurable via the corresponding tensor-product algebra.\footnote{We use the $\infty$-metric for
ease of proof when studying filtrations.  Other $p$-metrics or more general
product metrics might carry the whole theory, too, but we have not yet verified this.}
For any list $\tilde{A}$ representing a permutation of the set $A$,
then $\Val(\tilde{A})$ is the correspondingly ordered total product of all the measurable spaces of all attributes. 
At the other extreme, we equip the empty attribute list $[ ]$, of length $0$, with the
\emph{trivial value space} as $\Val([]) = \{*\}$, a singleton set.  

\begin{defn}[Set of Attribute Lists]
Let $\sS{A}$ denote the set of all attribute lists in $A$.
For each $n \geq -1$, let $\sS{A}_n\subset \sS{A}$ denote the set of all attribute lists of
length $n{+}1$.
$\sS{A}$ is a small category.
Using the notation from Appendix~\ref{sec:sset}, an object in $\sS{A}$ is a function $T:\mathbf{n} \to A$.
The case  $n=-1$, giving the empty list $T=[ ]$, is allowed.
A morphism of attribute lists $T \to T'$ is 
given by $\ell:\mathbf{n}' \to \mathbf{n}$
(an order-preserving function, which is a morphism of $\mathbf{\Delta_a}$ as in
Appendix~\ref{sec:sset}) 
such that $T' = T \circ \ell$, which is natural for the commutative diagram
\eqref{eqn:pullback}.
\begin{equation}
\begin{tikzpicture}[node distance=2cm]
\node (n) {$\mathbf{n}$};
\node[below of=n] (A) {$A$};
\node[right of=n] (np) {$\mathbf{n'}$};
\node[below of=np] (Ap) {$A$};
\draw[->] (np) to node[above] {$\ell$} (n); 
\draw[->] (Ap) to node[above] {$=$} (A); 
\draw[->] (n) to node[left] {$T$} (A); 
\draw[->] (np) to node[right] {$T'$} (Ap); 
\end{tikzpicture}
\label{eqn:pullback}
\end{equation}
\end{defn}
In Section~\ref{sec:operations} it is shown that 
for $n\geq 0$, each $\sS{A}_n$ is equipped with face maps $d_i:\sS{A}_n \to
\sS{A}_{n-1}$ (by omission of the $i$th element as in Defn \ref{defn:FA}) and degeneracy maps
$s_i:\sS{A}_{n} \to \sS{A}_{n+1}$ (by repetition of the $i$th element as in Defn \ref{defn:DA}).
When omitting the trivial $-1$-level, $\sS{A}$ is the simplicial set whose
elements are  generated  by the permutations of $A$ via the face and degeneracy
maps.  Including the trivial $-1$-level, $\sS{A}$ is the augmented
simplicial set generated this way.  See Appendix~\ref{sec:sset}
for a summary of the standard definition of (augmented) simplicial sets.

For any attribute list $T$, let $\Msr(T)$ denote the space of finite measures
on $\Val(T)$.
A \emph{data table} is a pair $(T, \tau)$ for $\tau \in
\Msr(T)$ for any $T \in \sS{A}$.  
 Note that $\Msr([])\cong \mathbb{R}_{\geq 0}$, as a measure on
the singleton set $\Val([ ])$ is determined by the mass $M \geq 0$ of $\{*\}$.
A \emph{trivial data table} is any data table of the form $(T, \tau)$
where $T = []$  and $\tau = M \geq 0$ is a measure on the singleton set $\Val(
[ ]) = \{*\}$.
We sometimes abbreviate our notation for data tables from $(T,\tau)$ to $\tau$,
because any $\tau \in \Msr(T)$ is equipped with a domain (the measurable sets
in $\Val(T)$), so $T$ is understood in context.

In the first example alluded to in the introduction, we could have $T = [\text{age}, \text{height},
\text{weight}]$ with $\Val(\text{age}) = \{0,1,\ldots,150\}$ in integer years,
$\Val(\text{height}) = [0,500] \subset \mathbb{R}$ in centimeters,
and $\Val(\text{weight}) = [0,1000] \subset \mathbb{R}$ in 
kilograms, each with the standard metric.
The space $\Msr(T)$ would be the set of measures on the compact set
$\Val(T) \subset \mathbb{R}^3$ given by the product.
An attribute list $[\text{height},\text{height}]$ is also permissible, and might arise for example if heights were compared from two different sources (driver's license versus medical chart).

For practical purposes, because $\Val(T)$ is a compact metric space, one might
use the Radon--Nikodym theorem to write any $\tau \in \Msr(T)$ using a density
function with respect to the uniform\footnote{That is, the measure depends only
on $r$, for metric balls $B_r(x)$ of sufficiently small radius.} probability
measure on the compact set; however, for simplicity we use the language and
notation of measures instead of the language of functions and integrals.

\begin{defn}[Ambient Data Complex]
Given $A$, the \emph{ambient data complex} over $A$ is the set of all data tables,
\[\sS{X} = \{ (T, \tau)~:~T \in \sS{A},\ \tau \in \Msr(T)\}.\]
For $-1 \leq n$, let $\sS{X}_n = \{ (T,\tau) \in \sS{X} \,:\, T \in
\sS{A}_n,\ \tau \in \Msr(T)\}$.
Let $p:\sS{X} \to \sS{A}$ denote the forgetful map $p:(T,\tau) \mapsto T$.
\end{defn}

Theorem~\ref{thm:sset} shows that the ambient data complex is a simplicial set
(augmented when including $\sS{X}_{-1}$) with faces given by the
marginalization integrals (Definition \ref{defn:FT}) and degeneracies given by
Dirac diagonalizations or intersections (Definition \ref{defn:DT}). The ambient data complex
$\sS{X}$ is a small category, whose morphisms are generated by faces and
degeneracies. The forgetful functor $p$ is a simplicial map between the small
categories $\sS{X}$ and $\sS{A}$.

\begin{defn}[Data Subcomplex]
Given an ambient data complex $p:\sS{X} \to \sS{A}$, a \emph{Data Subcomplex}
is a subset/subcategory
$\sS{X}' \subseteq \sS{X}$ that is closed under the face and degeneracy maps
defined in \ref{defn:FT} and \ref{defn:DT}.
Because $p$ is a simplicial map, the attribute base \[\sS{A}' = p(\sS{X}') = \{
T \in \sS{A}_n \,:\, \exists n \geq -1,\  \exists (T,\tau) \in
\sS{X}'_n\}\] is a simplicial subset of $\sS{A}$.
\end{defn}

\begin{defn}[Finitely Generated]
A data subcomplex $p:\sS{S} \to \sS{B}$ of an ambient $p:\sS{X}\to\sS{A}$
is said to be \emph{finitely generated} iff there is a finite set
$\{(T_1,\tau_1), \ldots, (T_K,\tau_K)\} \subset \sS{S}$ such that every data
table in $\sS{S}$ is obtained from this finite set via face and degeneracy maps.
We write $\sS{S} = \pair{(T_1,\tau_1), \ldots, (T_K,\tau_K)}$ 
or just $\sS{S} = \pair{\tau_1,\ldots,\tau_K}$.
\end{defn}

\begin{defn}[Closed under Permutation]
A subset $\sS{B}$ of $\sS{A}$
is said to be \emph{closed under permutation} iff for any $T \in
\sS{S}$ with $\len(T) = n{+}1$ and for any permutation (that is, bijection)\footnote{Note that a
nontrivial permutation is \emph{not} morphism in $\mathbf{\Delta_a}$.} 
$\varsigma:\{0,\ldots,n\}\to \{0,\ldots,n\}$, there exists
$\tilde{T} = T \circ \varsigma \in \sS{B}$.
A data subcomplex $p:\sS{S} \to \sS{B}$ of an ambient $p:\sS{X} \to \sS{A}$
is said to be \emph{closed under permutation} iff for any $(T,\tau) \in
\sS{S}$ with $\len(T) = n{+}1$ and for any permutation
$\varsigma:\{0,\ldots,n\}\to \{0,\ldots,n\}$, there exists
$(\tilde{T},\tilde{\tau}) \in \sS{S}$ 
and such that the measure $\tilde{\tau}$ is evaluated on the basis sets
$U_{\varsigma(0)} \times \cdots \times U_{\varsigma(n)}$ of the Borel algebra
of $\Val(\tilde{T})$ by
\[ \tilde{\tau}(U_{\varsigma(0)} \times \cdots \times U_{\varsigma(n)})=
 \tau(U_0 \times \cdots \times U_n).\]
\end{defn}

\begin{rmk}
Actual database merging problems encountered in real-life situations such as Problems \ref{GPtest}--\ref{GPmetamerge} always present themselves as Finitely
Generated Data Subcomplexes, because there is some finite set of database
tables or spreadsheets under consideration. 
The face and degeneracy maps provide the
logical relations between these tables that allow or prevent joining.
Real-life situations are also closed under permutation; because,
the ``\texttt{SELECT * FROM \dots}'' clause in SQL allows the database engineer
to re-order the columns of any table.
In our earlier example, a data table given by listing patients' age-by-height-by-weight might be
permuted to height-by-weight-by-age simply by reordering the columns of the
spreadsheet.
\end{rmk}

\textbf{Notational Note!}
We always use $p:\sS{X} \to \sS{A}$ to refer to an ambient data complex.  
We use either $p:\sS{X}' \to \sS{A}'$ or $p:\sS{S} \to \sS{B}$ to refer to a data subcomplex of $p:\sS{X} \to \sS{A}$.  
We tend to use $p:\sS{S} \to \sS{B}$ when we imagine that this data subcomplex
came from an actual data merging problem (so it is likely to be finitely generated and closed under permutation); however, we state explicitly these conditions when they are required for a result.
When the projection $p$ and the attribute simplicial sets $\sS{A}, \sS{B}$ are not used in a statement, we omit them and write ``a data subcomplex $\sS{S}$ of an ambient $\sS{X}$.'' 

\subsection{Morphisms of Data Tables}
\label{sec:operations}
This section establishes notation for common operations and proves that
$\sS{A}$ and $\sS{X}$ are simplicial sets, establishing that they are small
categories with morphisms that are well-understood in language of measures.

\begin{defn}[Face of Attribute List]
\label{defn:FA}
The face map on attribute lists, $d_i:\sS{A}_n \to \sS{A}_{n-1}$, is defined as
omission of the $i$th entry $a_i$ in $T=[a_0,\ldots,a_i,\ldots,a_n]$, so 
\[d_i [a_0,\ldots,a_{i-1},a_i,a_{i+1},\ldots,a_n] =
[a_0, \ldots, a_{i-1}, a_{i+1}, \ldots, a_n].\]
\end{defn}
\begin{rmk}[Categorical Interpretation]
In Definition~\ref{defn:FA}, $d_i T= T\circ d^i = (d^i)^* T$, where
$d^i:\mathbf{n-1} \to \mathbf{n}$ is
the co-face monomorphism in $\mathbf{\Delta_a}$, as in Appendix~\ref{sec:sset}.
\end{rmk}

\begin{defn}[Face of Data Table]
\label{defn:FT}
For a data table $(T,\tau) \in \sS{X}_n$ with $T = [a_0, \ldots, a_i, \ldots,
a_n]$, let $d_i(\tau) \in \Msr(d_i(T))$ be
the measure evaluated on the basis sets $U_0 \times \cdots \times U_{i-1} \times
U_{i+1} \times \cdots \times U_n$ of the Borel algebra on $\Val([a_0, \ldots, a_{i-1},
a_{i+1}, \ldots, a_n])=\Val(d_i T)$ as
\[ 
d_i(\tau)(U_0 \times \cdots U_{i-1} \times U_{i+1} \times \cdots \times U_n)
\defeq
\tau(U_0\times \cdots \times U_{i-1} \times  \Val(a_i) \times U_{i+1}\times \cdots \times U_n).\]
This is the measure obtained by marginalization to omit the $i$th factor, which
could also be written as $d_i \tau \defeq \int_{\Val(a_i)}\tau$.  
Let $d_i(T,\tau)\defeq(d_iT,d_i\tau)$, which is well-defined in $\sS{X}_{n-1}$.
\end{defn}

In our earlier example with individual patients as atomic point-masses,
the face map $d_0$ from age-by-height-by-weight to height-by-weight represents 
deleting the age column of the spreadsheet, and allowing new duplicate entries
to add (that is, integrate) measure.

Face maps can be applied multiple times, and the following lemma provides the
desired re-ordered ``commutation'' property.   
For attribute lists the proof is immediate; for data tables it is the
Fubini--Tonelli Theorem applied to the measures. 
\begin{lemma}[Fubini--Tonelli Theorem]
For any $i<j$, $d_i \circ d_j = d_{j-1} \circ d_i$.
\label{lem:FT}
\end{lemma}
\begin{defn}[Degeneracy of Attribute List]
\label{defn:DA}
The degeneracy map on attribute lists, $s_i:\sS{A}_n \to \sS{A}_{n+1}$, is defined
as repetition of the $i$th entry $a_i$ in $T=[a_0, \ldots, a_i, \ldots, a_n]$, so $s_i T =
[a_0, \ldots, a_i, a_i, \ldots, a_n]$.
\end{defn}

\begin{rmk}[Categorical Interpretation]
In Definition~\ref{defn:DA}, $s_i T= T\circ s^i = (s^i)^* T$, where
$d^i:\mathbf{n+1} \to \mathbf{n}$ is
the co-degeneracy epimorphism in $\mathbf{\Delta_a}$, as in Appendix~\ref{sec:sset}.
\end{rmk}

\begin{defn}[Degeneracy of Data Table]
\label{defn:DT}
For a data table $(T,\tau) \in \sS{X}_n$, let $s_i(\tau) \in \Msr(s_i(T))$ be
the measure evaluated on the basis sets $U_0 \times \cdots U_i \times U_i'
\times \cdots U_n$ of the Borel algebra on $\Val([a_0, \ldots, a_i,
a_i, \ldots, a_n]) = \Val(s_i T)$ as
\[ 
s_i(\tau)(U_0 \times \cdots U_i \times U_i' \times \cdots \times U_n)
\defeq
\tau(U_0\times \cdots \times (U_i \cap U_i') \times \cdots \times U_n).\]
Then, $s_i(T,\tau)=(s_iT,s_i\tau)$ is well-defined in $\sS{X}_{n+1}$.
\end{defn}
If the measure is expressed as a density function via the Radon--Nikodym theorem, then this is the Dirac-delta
\[s_i(\tau)(x_0, \ldots, x_i, x_i', \ldots, x_n) 
\defeq \tau(x_0, \ldots, x_i, \ldots, x_n) \delta(x_i, x_i'),\]

\begin{thm}[Simplicial Sets]
Let $\sS{X}$ be the ambient data complex over an attribute set $A$.
For any $(T,\tau) \in \sS{X}_n$, consider the face maps $d_i(T,\tau)$ and degeneracy maps 
$s_i(T,\tau)$ as in the definitions above.
Then
\begin{enumerate}
\item $d_i \circ d_j = d_{j-1} \circ d_i$, if $i<j$;
\item $d_i \circ s_j = s_{j-1} \circ d_i$, if $i<j$;
\item $d_j \circ s_j = d_{j+1} \circ s_j= \mathrm{id}$;
\item $d_i \circ s_j = s_j \circ d_{i-1}$, if $i > j+1$; and 
\item $s_i \circ s_j = s_{j+1} \circ s_i$, if $i \leq j$.
\end{enumerate}
Moreover, the forgetful map $p:\sS{X}\to\sS{A}$ commutes with $d_i$ and $s_i$.
That is, $\sS{X}=(\sS{X}_n,d,s)$ and $\sS{A} = (\sS{A}_n,d,s)$ are augmented
simplicial sets as in Lemma~\ref{lem:generate}.   They are simplicial sets when omitting the trivial
$\sS{X}_{-1}$ and $\sS{A}_{-1}$.
\cite[Defn 3.2]{Friedman2008}\cite[Eqn (1.3)]{Seriesa}\cite[Defn 1.1]{May1967}
\label{thm:sset}
\end{thm}
\begin{proof}
This is direct with no surprises, by working on the Borel basis sets $U_0\times\cdots\times U_n$ for
$\Val([a_0,\ldots,a_n])$. The $d_id_j$ condition was already seen as Fubini--Tonelli.
\end{proof}

\subsection{Inclusions, Merges, and Joins}
\label{sec:IMJ}
We now establish\footnote{We are particularly indebted to Tony Falcone for technical discussions that motivated the formalism in this subsection.} additional operations (inclusion, sum, merge, join) that are special
to $\sS{A}$ and $\sS{X}$ and do not apply to general simplicial sets.

\begin{defn}[Attribute Inclusions]
An \emph{attribute inclusion} 
\[[a_0, a_1, \ldots, a_{n'}] \into  [b_0, b_1, \ldots, b_{n}]\] is given by
a map
$ \iota:\{0,\ldots, n'\} \to \{0,\ldots,n\}$
such that
\begin{enumerate}
\item $\iota(i) \leq \iota(j)\ \text{if and only if}\ i \leq j\
\text{(order-preserving),}$
\item $a_i  = b_{\iota(i)}\ \text{(compatible), }$
\item $\iota$ is one-to-one (implying $n'\leq n$),
\end{enumerate}
Although $\iota$ itself is a map of index sets,  we use the compatibility
property to overload notation and write 
$\iota:[a_0,\ldots,a_{n'}]\into [b_0,\ldots,b_n]$.
\end{defn}
\begin{rmk}[Categorical Interpretation]
An attribute inclusion is a morphism $T \to T'$ in the category $\sS{A}$ such that 
$T'=T\circ \iota = \iota^* T$ where $\iota:\mathbf{n'} \to \mathbf{n}$ is a monomorphism in
$\mathbf{\Delta_a}$.  We overload notation (that is, omit the upper-star) and write
$\iota:T'\into T$.
The functor $\mathbf{\Delta_a} \to \sS{A}$ is contravariant, so attribute
inclusions are actually epimorphisms in $\sS{A}$; however, it is reasonable to
call them ``inclusions'' because the $\mathbf{n'}$-ordered multiset
$T'(\mathbf{n'})$ is an ordered subset of the $\mathbf{n}$-ordered multiset $T(\mathbf{n})$. One could avoid this overloaded notation by
working in the opposite category, but we decline to add another layer of
notation since the meaning is always clear in context.
\end{rmk}

\begin{ex}
Consider $A = \{a,b,c,d\}$.
Example attribute lists\footnote{The fact that these lists are in alphabetical order is merely aesthetic,
and is not required in the definition of an attribute list.}
 are $T' = [a,a,a,c,d]$ and $T=[a,a,a,a,a,b,c,c,d]$.
There are 20 possible inclusions $T' \into T$, which are obtained by choosing the
ordered image of the $a$'s and $c$'s.  One possible inclusion 
is 
\begin{equation}
\iota=\left\{ 
(0 \mapsto 0), 
(1 \mapsto 1),
(2 \mapsto 3),
(3 \mapsto 7),
(4 \mapsto 8)
\right\},
\end{equation}
which can be summarized as $\iota=[0,1,3,7,8]$.
We can abbreviate this by decorating the entries in $T$ that are included from
$T'$, 
\begin{equation}
\iota:[a,a,a,c,d]\into
[\underline{a},\underline{a}, a, \underline{a}, a, b, c, \underline{c},\underline{d}].
\end{equation}
\end{ex}

\begin{lemma}[Quotient Inclusion]
For any attribute inclusion $\iota:T' \into T$, there is an attribute list
$T/\iota$
(called the \emph{quotient}) that enumerates the entries of $T$ that are not in the
image of $\iota$.  This enumeration equips the quotient with an attribute inclusion
$\iota^c: (T/\iota) \into T$, and $\iota^c$ corresponds to the complimentary monomorphism
from Lemma~\ref{lem:compliment}.
\label{lem:quotient}
\end{lemma}
\begin{ex}
Consider the earlier example of an attribute inclusion.
The quotient list is $T/\iota = [a,a,b,c]$.   
The quotient inclusion is 
\begin{equation}
\iota^c:[a,a,b,c] \into [a, a, \underline{a}, a, \underline{a}, \underline{b},
\underline{c}, c, d].
\end{equation}
\end{ex}

The next lemma and corollary make clear that face maps and attribute inclusions are related
tightly.

\begin{lemma}
Any face map $d_i:T\to d_iT$ in $\sS{A}$ is equipped with an attribute inclusion $d_iT \into T$ defined
by the index function that skips $i$, namely the co-face monomorphism $d^i$ in
$\mathbf{\Delta_a}$.  Its quotient inclusion is the index function $\{0\} \to
\{0,\ldots,n\}$ by $0\mapsto i$ that gives $[a_i]\into T$.
\end{lemma}

\begin{cor}
For any attribute inclusion $\iota:T' \into T$, there is a sequence\footnote{The backwards ordering here is intentional.
Because of the indexing situation and Lemma~\ref{lem:FT}, it is simpler to remove attributes from the
end. To remove an entire list, one could write $d_0 d_1 \cdots d_n T$ or
$d_0d_0\cdots d_0 T$, because $d_0$ is like ``pop'' on the front of the list.}  
of face maps $d_{j_0},\ldots,d_{j_k}$ such that $d_{j_0}\cdots d_{j_k} T = T'$
and such that the attribute inclusion induced by the sequence of face maps is
$\iota$. Moreover, any permutation of this sequence obtained by re-indexing
the face-maps according to Lemma~\ref{lem:FT} is equivalent.  If $j_0 \leq
\cdots \leq j_k$, then $j$ is the index function for $\iota^c$, the quotient inclusion.
\label{cor:include}
\end{cor}
\begin{rmk}
In light of Theorem~\ref{thm:sset}, Corollary~\ref{cor:include} is a partial
version of Lemma~\ref{lem:generate}, which says face maps and
degeneracy maps generate all the morphisms in a simplicial set. This is because
the co-face and co-degeneracy maps in $\mathbf{\Delta_a}$ generate all 
order-preserving maps.
\end{rmk}

Attribute inclusions provide surjections on value spaces and measures,
according to the following ``contravariant'' definition.
\begin{defn}[Reduction]
Consider an attribute inclusion $\iota:T' \into T$, where
$T'=[a_0,\ldots,a_{n'}]$
and $T= [b_0,\ldots,b_n]$.
Write an element of $\Val(T)$ as $(x_0,\ldots, x_n)$ where $x_i \in \Val(b_i)$
and write an element of $\Val(T')$ as $(y_0,\ldots,y_{n'})$ where $y_j \in
\Val(a_j)$.
Define the surjective function $\red_{\iota}:\Val(T) \onto
\Val(T')$ by  
\[  (x_0, \ldots, x_{n}) \mapsto  (y_0,\ldots,y_{n'}) =
(x_{\iota(0)},\ldots,x_{\iota(n')}).\]
Similarly, define the surjective function $\red_{\iota}:\Msr(T) \onto \Msr(T')$ 
by sequential application of face maps according to the previous corollary:
For any $\tau \in \Msr(T)$, let 
\[  \red_{\iota}(\tau)(U_{0}\times \cdots \times U_{n'})
\defeq \tau(W_0\times\cdots\times W_n),~\text{for}~W_i = 
\begin{cases}
U_{j},&\text{if $i=\iota(j)$}\\
\Val(a_{i}),&\text{otherwise}.
\end{cases}
\]
That is, $\red_{\iota}\tau$ is the measure on $\Val(T')$ obtained by
marginalizing $\tau$ to remove the factors specified by $\iota^c$.
\label{defn:red}
\end{defn}
When the attribute inclusion $\iota:T' \to T$ is understood from context, we
abuse notation and write $\red_{T'} \tau$ instead of $\red_{\iota} \tau$.
Note that $\tot \tau = (d_0 \cdots d_n) \tau = \tau(\Val(T)) = \int_{\Val(T)} \tau$, so we use this notation as
shorthand for ``the total integral of a measure.''

\begin{defn}[Sum of Attribute Lists]
Given attribute lists $T_1$ and 
$T_2$ in $\sS{A}$, define 
$T_1 \oplus T_2$ as the attribute list obtained by 
concatenating $T_1$ and $T_2$. 
\label{def:listsum}
\end{defn}

Note that $T_1 \oplus T_2$ and $T_2 \oplus T_1$ are related by a permutation,
which (excepting the identity permutation) does \emph{not} correspond to a morphism in the
categories $\sS{A}$ or $\mathbf{\Delta_a}$.
The concatenation process provides specific attribute inclusions
$T_1 \into T_1 \oplus T_2$ and $T_2 \into T_1 \oplus T_2$.
More generally, for attribute inclusion $\iota:T' \to T$ as in
Lemma~\ref{lem:quotient}, it is true that $T$ and $T' \oplus (T/\iota)$ are
related by a permutation; because,
the concatenation provides inclusions $T' \into T$ and $(T/\iota)\into T$ that 
may \emph{not} be the original $\iota$ and $\iota^c$.
On the other hand, for any sum $T = T_1 \oplus T_2$, it is true that $T_2$ is the
quotient of $T$ by the concatenation-induced inclusion of $T_1$, and
vice-versa.

Definition~\ref{def:listsum} implies $\Val(T_1 \oplus T_2) = \Val(T_1) \times \Val(T_2)$ and
$\Msr(T_1 \oplus T_2)$ are well-defined.
But, beware of multivariable calculus: $\Msr(T_1 \oplus T_2)
\supsetneq \Msr(T_1)\times \Msr(T_2)$,
as not every measure on a product space is an elementary product of measures!

\begin{ex}
Consider $T_1 = [a, a, a, c, d]$ and $T_2 = [a,a,b,c]$.
Then their sum is $T_1 \oplus T_2 = [a,a,a,c,d,a,a,b,c]$.  The concatenation
is equipped with inclusions 
\begin{equation}
\begin{split}
\iota_1 =\iota_2^c: [a,a,a,c,d]&\into
[\underline{a},\underline{a},\underline{a},\underline{c},\underline{d},a,a,b,c]\\
\iota_2=\iota_1^c: [a,a,b,c]&\into
[a,a,a,c,d,\underline{a},\underline{a},\underline{b},\underline{c}].
\end{split}
\end{equation}
\end{ex}

\begin{defn}[Permutation Notation]
Suppose that $T_{12}$, $T_{1}$, and $T_{2}$ are attribute lists such that 
$\varsigma(T_{1} \oplus T_{2}) = T_{12}$ for a permutation
$\varsigma$.
Then $\iota=\varsigma|_{T_1}:T_{1}\to T_{12}$ and 
$\iota^c=\varsigma|_{T_2}:T_{2}\to T_{12}$ are complimentary attribute inclusions.
If the permutation or attribute inclusions are well-known in context, then for
any subsets $U_1 \subseteq \Val(T_1)$ and $U_2 \subseteq \Val(T_2)$, let $U_1
\ptimes U_2 \in \Val(T_{12})$ denote the subset for which elements $x_1 \times
x_2 \in U_1 \times U_2 \subseteq \Val(T_1 \oplus T_2)$ and
$x_1 \ptimes x_2 \in U_1 \ptimes U_2 \subseteq \Val(T_{12})$ correspond with
respect to the $\varsigma$-permuted indices. 
\end{defn}
Because Lemma~\ref{lem:mergeind} provides an ordered form of the inclusion--exclusion
principle, we can define an indexed form of the inclusion--exclusion principle.

\begin{defn}[Merge of Attribute Lists]
Suppose $T_{0}, T_{01}, T_{02} \in \sS{X}$, and that $\iota_{01}: T_{0} \into
T_{01}$ and $\iota_{02}:T_{0} \into T_{02}$ are attribute inclusions.
Define $\Merge(T_{01},T_{02},\iota_{01} \sim \iota_{02})$ 
as the attribute list obtained by performing the index merge specified by
Figure\ref{mergeind} as in Lemma~\ref{lem:mergeind}; this merge concatenates sublists spliced between the
entries aligned by $\iota_{01} \sim \iota_{02}$.  Writing $T_{012}$ for
$\Merge(T_{01},T_{02},\iota_{01} \sim \iota_{02})$, Diagram~\ref{eq:split}
becomes a diagram of attribute inclusions. 
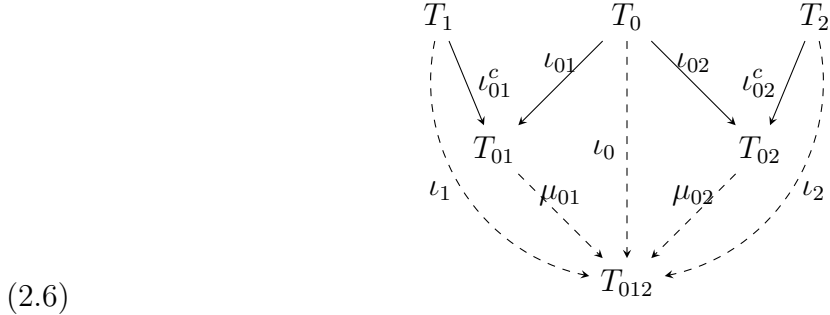
\begin{equation}
\begin{tikzpicture}[node distance=2.5cm]
\node (k0) {$T_0$};
\node[left of=k0] (k1) {$T_1$};
\node[right of=k0] (k2) {$T_2$};
\node[below left of=k0] (k01) {$T_{01}$};
\node[below right of=k0] (k02) {$T_{02}$};
\node[below right of=k01] (k012) {$T_{012}$};
\draw[->] (k0) to node[above] {$\iota_{01}$} (k01); 
\draw[->] (k0) to node[above] {$\iota_{02}$} (k02); 
\draw[->] (k1) to node[right] {$\iota_{01}^c$} (k01); 
\draw[->] (k2) to node[left] {$\iota_{02}^c$} (k02); 
\draw[dashed, ->] (k01) to node[above] {$\mu_{01}$} (k012); 
\draw[dashed, ->] (k02) to node[above] {$\mu_{02}$} (k012); 
\draw[dashed, ->] (k0) to node[left] {$\iota_{0}$} (k012); 
\draw[dashed, ->, bend right=45] (k1) to node[left] {$\iota_{1}$} (k012); 
\draw[dashed, ->, bend left=45] (k2) to node[right] {$\iota_{2}$} (k012); 
\end{tikzpicture}
\label{eq:splitT}
\end{equation}
\label{defn:merge}
\end{defn}
Note that the choice of ordering in Definition~\ref{defn:merge} and 
Figure~\ref{mergeind} is partially arbitrary.  In particular, one may draw an
equivalent diagram with any choice of interleaving pattern, as long as the $T_0$
entries remain fixed. 
However, this choice is irrelevant, as the theory developed in
Section~\ref{sec:homotopy} will encompass all allowable permutations.
Regarding the permutation notation introduced earlier,
for any Borel sets $U_0 \subseteq \Val(T_0)$,
$U_1 \subseteq \Val(T_1)$, and
$U_2 \subseteq \Val(T_2)$, we write 
$U_{0}\ptimes U_{1} \ptimes U_{2} \subseteq \Val(T_{012})$ for the appropriately permuted copy of the
set $U_{0} \times U_{1} \times U_{2}$ in $\Val(T_{0}\oplus T_{1} \oplus
T_{2})$, since the definition and algorithm give a well-defined permutation.
This $\ptimes$ notation is required in Theorem~\ref{thm:ft} and elsewhere.

\begin{ex}
Compare this to Example~\ref{ex:ind}.
Consider $A = \{a,b,c,d\}$.
Consider attribute lists  $T_{01} = [a,a,a,a,b,c]$ and $T_{02} = [a,a,b,b,d]$,
and $T_{0} = [a,b]$ with the attribute inclusions $\iota_{01} = [1,4]$ and
$\iota_{02}=[1,3]$. 
Visually, the merged indexing means
\[
\begin{split}
\iota_{01}:[a,b] &\mapsto [a,\underline{a},a,a,\underline{b},c] \\
\iota_{02}:[a,b] &\mapsto [a,\underline{a},b,\underline{b},d] \\
&\text{yields}\\
\iota_{0}:[a,b] &\mapsto [a, a, \underline{a}, a, a, b, \underline{b}, c,
d]\\
\mu_{01}:[a,a,a,a,b,c] &\mapsto
[\underline{a}, a, \underline{a}, \underline{a}, \underline{a}, b,
\underline{b}, \underline{c}, d]\\
\mu_{02}:[a,a,b,b,d] &\mapsto
[a, \underline{a}, \underline{a}, a, a, \underline{b}, \underline{b}, c,
\underline{d}]
\end{split}
\]
\end{ex}

Our choice of ordering in $\Merge( )$ provides that the trivial merge $\Merge(T_{01},T_{02},[ ]) = T_{01} \oplus T_{02}$
is the sum from Definition~\ref{def:listsum}.

As with Definition~\ref{def:listsum}, the attribute list
$\Merge(T_{01},T_{02}, \iota_{01}\sim\iota_{02})$ is well-defined regardless of the preference of
$T_{01}$ versus $T_{02}$ and regardless of the indices specified by
$\iota_{01}$ and $\iota_{02}$.  This list is identical to the list obtained by
constructing the sum $T_{01}\oplus T_{02}$ then applying face maps to remove the
image of $d_{\iota_{02}(i)}$ for each $i$ indexing $T_0$.
But, again beware that the partitioned merge-sort construction equips
$T_{012}$ with specific attribute inclusions
$T_{01} \into  T_{012}$ and $T_{02}\into T_{012}$ such that
the composed attribute inclusion $T_{0} \into T_{012}$ is
well-defined through both compositions.  In general, these inclusions are not
the same as the inclusions obtained through the ``sum and face''
construction.

\begin{lemma}[Decomposition of Merged Lists]
Suppose $T_{0}, T_{01}, T_{02} \in \sS{X}$, and that $\iota_{01}: T_{0} \into
T_{01}$ and $\iota_{02}:T_{0} \into T_{02}$ are attribute inclusions.
Let $T_{012}$ denote 
$\Merge(T_{01},T_{02},\iota_{01}\sim\iota_{02})$.
Let $\iota_{01}^c:T_{1}\into T_{01}$ and $\iota_{02}^c:T_{2}\into T_{02}$ denote the
complements of these inclusions, so $T_{1} \defeq T_{01}/\iota_{01}$ and $T_{2}
\defeq T_{02}/\iota_{02}$.

Then $T_{012}$ is partitioned by the inclusions $\iota_0:T_0 \into T_{012}$,
$\iota_1:T_1 \into T_{012}$, and $\iota_2:T_2\into T_{012}$.   That is,
$\Val(T_{012}) = \Val(T_0)\ptimes \Val(T_1)\ptimes \Val(T_2)$ is a permutation of $\Val(T_0)\times \Val(T_1) \times
\Val(T_2)$, with ordering of projections determined by the inclusion and
quotient maps, $\red_{\iota_0}$, $\red_{\iota_1}$, and $\red_{\iota_2}$.
\label{lem:decomp}
\end{lemma}

\subsection{Data Sections}
\label{sec:sec}
Because the forgetful map $p:\sS{X} \to \sS{A}$ acts like a projection, it allows a notion of section.
\begin{defn}[Data Section]
Consider a data subcomplex $p:\sS{X}'\to\sS{A}'$ of an ambient $p:\sS{X}\to\sS{A}$.
A \emph{data section} is a natural\footnote{Natural means that it respects the
face and degeneracy maps, as in \eqref{eqn:natsS} and Lemma~\ref{lem:generate}.} map $\sigma:\sS{A}'\to\sS{X}'$ such that $p \circ \sigma = 1_{\sS{A}'}$.
\end{defn}

\begin{rmk}
In Section~\ref{sec:FO}, data sections will be specified as $\sigma:\sS{A}'_n \to \sS{X}'_n$, on a
single level of the simplicial-set grading, where the other levels are inferred
by the face and degeneracy maps.  This omits all nondegenerate elements of level $n+1$, so is interpreted as a section on the $n$-skeleton.
\end{rmk}

The following definition captures a condition describing data subcomplexes
that are ``as compatible as possible.''
\begin{defn}[Well-Aligned]
A data subcomplex $\sS{X}'$ of an ambient $\sS{X}$ is called
\emph{well-aligned} if: for all $(T_{01},\tau_{01}), (T_{02},\tau_{02}) \in
\sS{X}'$ and all $T_{0}$ with attribute inclusions $\iota_{01}:T_{0}\into
T_{01}$ and $\iota_{01}:T_{0} \into T_{02}$, there exists $(T_0, \tau_0) \in \sS{X}'$ with  \[ \red_{\iota_{01}}\tau_{01} = \tau_{0} = \red_{\iota_{02}}\tau_{02}.\]
\label{def:well}
\end{defn}

The next lemma shows that well-aligned data subcomplexes in this theory play
the role of ``submanifolds transverse to the fiber'' from classical bundle theory and of ``holonomic submanifolds'' in geometric PDE theory.  That is, they represent local sections.
\begin{lemma}
Suppose that $p:\sS{X}' \to \sS{A}'$ is a data subcomplex of an ambient data complex $p:\sS{X} \to \sS{A}$ such that $\sS{X}'$ contains a nontrivial data table.  The following are equivalent:
\begin{enumerate}
\item $\sS{X}'$ is well-aligned.
\item There is a data section $\sigma:\sS{A}' \to \sS{X}'$ such that $\sigma(\sS{A}') = \sS{X}'$.
\item $p:\sS{X}' \to \sS{A}'$ is a simple cover via the isomorphism $p$.
\end{enumerate}
\end{lemma}
\begin{proof}
(2) implies (1): Note that well-alignedness is implied by the commutation of $\sigma$ with the face maps.

(1) implies (2): 
The case of $T_{0} = []$ implies that all data tables in $\sS{X}'$
have the same mass, $M$, which is non-zero since $\sS{X}'$ contains at least one non-trivial table.  The
case of $T_{01} = T_{02} = T_{0} = T$ implies that each $T \in \sS{A}'$ admits
exactly one $(T,\tau) \in \sS{X}'$. 

It is immediate that (2) and (3) are equivalent.
\end{proof}

\begin{rmk}
A database engineer would appreciate a database system that could be described
as a well-aligned data subcomplex, because for each list of
columns present within any combination of the given tables, there is only one possible table; that is, for each $T$ there is exactly one
$(T,\tau)$.  Compare the well-aligned condition to the space of joins, Defn~\ref{def:joinspace}.
Note too that well-aligned implies finitely generated. 
(Not every finitely-generated data subcomplex is well-aligned, as it could have
multiple data tables over the same attribute lists.)
Moreover, if $\sS{X}'$ is well-aligned (and contains a nontrivial data table),
then all data tables can be re-scaled by their shared mass $M$ to yield probability measures.
\end{rmk}

In our definitions above, the set $A$ of attributes is finite, and each level $\sS{A}_n$ of the simplicial set $\sS{A}$ is finite.
Therefore, for any simplicial subset $\sS{A}' \subseteq \sS{A}$, we can consider the finite graph whose vertices are the 0-cells of $\sS{A}'_0$ (singleton attribute lists) and whose edges are the 1-cells $\sS{A}'_1$ (including loops, as degenerate 1-cells like $[a,a]$).

\begin{defn}[Connected]
Suppose that $p:\sS{X}' \to \sS{A}'$ is a data subcomplex of an ambient data complex $p:\sS{X} \to \sS{A}$ such that $\sS{X}'$ contains a nontrivial data table.
The simplicial set $\sS{A}'$ of attributes is called \emph{connected} if the finite graph with vertices $\sS{A}'_0$ and edges $\sS{A}'_1$ is connected.
\end{defn}

\begin{defn}[Path-connected]
A data subcomplex $\sS{X}'$ of an ambient $\sS{X}$ is called \emph{path-connected} if
for any attributes $a, b$  in $\sS{A}'$, there is sequence $(T_{01},\tau_{01}), (T_{12},\tau_{12}), \ldots, (T_{k-1,k},\tau_{k-1,k})$ in $\sS{X}'$ such that $a \in T_{0,1}$ and $b \in T_{k-1,k}$ and for all $i=1\ldots,k$
there is an attribute list $T_{i} \neq []$ equipped with inclusions $T_{i}\into T_{i-1,i}$ and $T_{i}\into T_{i,i+1}$ such that
\[\red_{T_i}\tau_{i-1,i} =\red_{T_i}\tau_{i,i+1}.\]
\end{defn}

\begin{lemma}
Suppose $\sS{A}'$ is connected as a simplicial set.
If $\sS{X}'$ is well-aligned, then $\sS{X}'$ is path-connected.
\end{lemma}

With the language of simplicial sets, we can now re-state our original
motivating questions.  The remaining sections of this document construct a
precise way to answer these questions, and ensure that the notion of
``distance'' is well-defined.
An appropriate notion of distance appears in Defn~\ref{def:wasserstein}.
When all the definitions and lemmas are in place, these problems are answered by
the Obstruction Cocycle in Defn~\ref{def:cocycle}.

\begin{prob}[Testing Problem, bis]
Consider a data subcomplex $p:\sS{S}\to \sS{B}$ of an ambient $p:\sS{X}\to\sS{A}$.
Given a data section $\sigma^+:\sS{A}_{n+1} \to \sS{X}_{n+1}$ of the form 
$\sigma^+:T^+ \mapsto (T^+,\sigma^+)$ for $T \in \sS{A}_{n+1}$, 
is it true that $\partial \sigma^+$ lies entirely within $\sS{S}_{n}$?
If not, what is the distance from $\partial \sigma^+$ to $\sS{S}_n$ in
$\sS{X}_n$?
\label{GPtest2}
\end{prob}

\begin{prob}[Merging and Meta-Merging Problems, bis]
Consider a data subcomplex $p:\sS{S}\to \sS{B}$ of an ambient $p:\sS{X}\to\sS{A}$.
Suppose that there is a simplicial map $\sigma:\sS{B}_n \to \sS{S}_n$ of the
form $\sigma:T  \mapsto (T,\sigma)$.
Does there exist an extension $\sigma^+:\sS{A}_{n+1} \to \sS{X}_{n+1}$ of
$\sigma$, meaning $\partial \sigma^+(T^+) = \sigma(\partial T^+)$ for all $T^+
\in \sS{A}_{n+1}$ such that $\partial T^+ \in \sS{B}_n$.
If not, what is the minimal distance that would allow an approximate
extension?
\label{GPmerge2}
\end{prob}

\subsection{Homology}
\label{sec:chains}
We use the traditional definition of chains, summarized here to fix notation.
Fix\footnote{For practical reasons we use $R =
\Zb = \mathbb{F}_2$ in applications; however, chains are sensible for any
ring.} a ring $R$.
For $k\geq 0$, a $k$-chain $Y \in C_k(\sS{A},R) \defeq C_k(\sS{A},\mathbb{Z})\otimes R$ is a formal linear combination
\begin{equation}
Y = \sum_j r_j T_j,\quad r_j \in R,\ T_j \in \sS{A}_k,
\end{equation}
where negative coefficients indicate formally reversed orientation.
We define $(-1)$-graded chains as elements of the 1-dimensional $R$-module, 
$C_{-1}(\sS{A}) = R\cdot\{[ ]\} \cong R$. 
For $k \geq 0$, define the usual simplicial boundary operator $C_k(\sS{A},R) \to
C_{k-1}(\sS{A},R)$, as  
\begin{equation}
\begin{split}
\partial &\defeq \sum_{i=0}^k (-1)^i d_i\\
\partial:& [a_0, \ldots, a_k] \mapsto  \sum_{i=0}^k (-1)^i[a_0, \ldots, \hat{a_i},
\ldots, a_k].
\end{split}
\end{equation}
It is immediate that $\partial:C_k(\sS{A},R) \to C_{k-1}(\sS{A},R)$ satisfies
$\partial^2=0$, so homology $H_\bullet(\sS{A},R)$ is well-defined.

For $k \geq 0$, a $k$-chain $(Y,\psi) \in C_k(\sS{X},R) \defeq C_k(\sS{X},\mathbb{Z})\otimes R$ is a formal linear combination
\begin{equation}
(Y,\psi) = \sum_j r_j (T_j, \tau_j),\quad r_j \in R,\ T_j \in \sS{A}_k,\ \tau_j \in \Msr(T_j),
\end{equation}
where negative coefficients indicate formally reversed orientation.
We can also define $(-1)$-graded chains as 
$C_{-1}(\sS{X},R)$, the $R$-module generated by  $\Msr([ ]) = R_{\geq 0}$. 
Moreover, for any $(T,\tau) \in \sS{X}_k$, note that $r(T,\tau)$ and
$(T,r\tau)$ are formally distinct unless $r=1_R$; hence, the graded module $C_\bullet(\sS{S},R)$ is \emph{very large}, especially if $\Val(a)$ is infinite for any $a\in A$.
For $k\geq 0$, define the usual simplicial boundary operator $C_k(\sS{X},R) \to
C_{k-1}(\sS{X},R)$, as  
\begin{equation}
\begin{split}
\partial &\defeq \sum_{i=0}^k (-1)^i d_i\\
\partial:& ([a_0, \ldots, a_k],\tau) \mapsto  \sum_{i=0}^k (-1)^i([a_0, \ldots, \hat{a_i},
\ldots, a_k], \red_{[a_0, \ldots, \hat{a_i}, \ldots, a_k]}\tau).
\end{split}
\end{equation}

The next lemma is easy, but important; it means the usual notions of \emph{cycle}/\emph{closed} and
\emph{boundary}/\emph{exact} apply to chains in $\sS{X}$.
\begin{lemma}
$\partial:C_k(\sS{X},R) \to C_{k-1}(\sS{X},R)$ satisfies $\partial^2=0$.
In particular, the homology $H_\bullet(\sS{X},R)$ is well-defined. 
\end{lemma}
\begin{proof}
Suppose that $T = [a_0, a_1, \ldots, a_k]$ and that $\tau \in \Msr(T)$.
Recall that $\Val(T)$ is a product of the attributes' measurable spaces,
and by our definitions, the measure $(\Val(T),\tau)$ is
finite, therefore $\sigma$-finite, so the Fubini--Tonelli theorem holds.
In particular, Corollary~\ref{cor:include} shows that the reduction
$\red_{[a_0,\ldots, \hat{a_i}, \ldots, \hat{a_j},\ldots, a_k]}$ is symmetric,
so because the double-sum is alternating, all terms will cancel. 

For example, suppose $\tau = \tau_{0123}$ on $T = [a_0, a_1, a_2, a_3]$, where the
index shows which variables are still free.  Then
\begin{equation}
\begin{split}
\partial^2\tau_{0123} 
&= \partial(\tau_{123} - \tau_{023} + \tau_{013} - \tau_{012})\\
&= \left(\tau_{23} - \tau_{13} + \tau_{12}\right)
  -\left(\tau_{23} - \tau_{03} + \tau_{02}\right)
  +\left(\tau_{13} - \tau_{03} + \tau_{01}\right)
  -\left(\tau_{12} - \tau_{02} + \tau_{01}\right)\\
&= 0 \in C_{k-2}(\sS{X}).
\end{split}
\end{equation}
Note that it is irrelevant in this proof whether $T$ is a degenerate attribute list, as
repeated attribute value spaces are treated as distinct factors for the sake of
integration.
\end{proof}
Define the projection $p:C_k(\sS{X},R) \to C_k(\sS{A},R)$ by 
$p(T, \tau) \defeq T$ and extending by linearity.
\begin{lemma}
$\partial \circ p = p \circ \partial$
\end{lemma}
\begin{proof}
Suppose that $T = [a_0, a_1, \ldots, a_k]$ and that $\tau \in \Msr(T)$.
Then 
\begin{equation}
p \partial (T,\tau)
= p \sum (-1)^i d_i(T, \tau)
= \sum (-1)^i p(d_iT, d_i\tau)
= \sum (-1)^i d_iT
= \partial T = \partial p (T,\tau).
\end{equation}
\end{proof}

\begin{cor}
The induced homomorphism $[p]:H_\bullet(\sS{X},R) \to H_\bullet(\sS{A},R)$ is
well-defined.
\end{cor}

Similarly, the chain modules and homology are well-defined for any data
subcomplex $p:\sS{X}' \to \sS{A}'$ of an ambient $p:\sS{X}\to \sS{A}$.
\begin{cor}
If $p:\sS{X}'\to\sS{A}'$ is well-aligned, then there are canonical isomorphisms $C_\bullet(\sS{X}',R) \cong C_\bullet(\sS{A}',R)$ and $H_\bullet(\sS{X}',R) \cong H_\bullet(\sS{A}',R)$ induced by $p$.
\end{cor}

We are particularly interested in the case $R = \Zb$, so that 
a chain $C_\bullet(\sS{A},\Zb)$ (respectively $C_\bullet(\sS{X},\Zb)$) is interpreted as a set of
attribute lists (respectively, data tables), without any consideration for multiplicity or orientation.  It is therefore sensible to apply the condition \emph{well-aligned} to a chain $(Y,\psi) \in C_n(\sS{X}',\Zb)$, so that a well-aligned chain in $(Y,\psi) \in C_n(\sS{X},\Zb)$ can be interpreted equivalently to a section $\sigma:p(\sS{X}') \to \sS{X}'$, where $\sS{X}'$ is the data subcomplex generated by the elements of $\psi$.

\section{Homotopy as Joins}
\label{sec:homotopy}
In the previous sections, we established that a data complex is
equipped with simplicial homology, and framed data complexes as simplicial sets.
This section contains several payoffs for that effort.
First, Section \ref{sec:DJKC} builds to Theorem~\ref{thm:ft}, our first key result, which shows that the simplicial set language enables a connection between our framework 
and standard database engineering; later results show that the framework enables further insights into data merging problems that transcend standard database engineering.
Then, Section \ref{sec:SHDC} explores the simplicial homotopy of data complexes
and reframes Problems \ref{GPtest2} and \ref{GPmerge2} in the language of
obstruction theory for simplicial sets, as in \cite{Kan1958, May1967, Quillen1967,
Seriesa, Friedman2008}.

\subsection{Database Joins and the Kan Condition}

\label{sec:DJKC}

Recall these three standard definitions from the well-established theory of
simplicial sets, as in Appendix~\ref{sec:sset} and \cite{Kan1958, May1967, Quillen1967, Seriesa, Friedman2008}.

\begin{defn}[Simplex]
The standard $n$-simplex $\Delta^n$ is the simplicial set generated (via face and
degeneracy maps) by the ordered set $\mathbf{n}=\{0,\ldots,n\}$ in the simplex
category $\mathbf{\Delta}$.
\end{defn}

\begin{defn}[Horn]
The $k$th \emph{horn} $\Lambda^n_k$ of the $n$-simplex $\Delta^n$ is the
simplicial subset generated by the union of all the faces of $\Delta^n$ except the $k$th face.
\end{defn}
As is standard in the literature, we abuse notation slightly by
referring to both $\Delta^n \to X$ (which is an infinite collection of sets)
and $\mathbf{n} \mapsto x \in X_n$ (which is the generator of that collection)
as ``an $n$-simplex in $X$.''

Note!  
The (categorical) $n$-simplex $\Delta^n$ is \emph{not} the same as the
(topological) $n$-simplex $|\Delta^n|$.  The former is an infinite set of formal
objects in the simplex category; it has no notion of ``interior'' or
``continuity.''  The latter is a compact topological space obtained defined via
convex linear combinations in $\mathbb{R}^{n+1}$.
There is a relationship between their respective categories called
\emph{realization}, as discussed in \cite[\S3]{Quillen1967} and
\cite[Chap~I.2]{Seriesa}.

\begin{rmk}[Data Tables as Simplices]
A data subcomplex $\sS{X}'$ in an ambient $\sS{X}$ is an (augmented) simplicial set
by Theorem~\ref{thm:sset}.  Thus, a data table $(T,\tau)$ with $T =
[a_0,\ldots,a_n]$ 
can be seen as (the generator of) an $n$-simplex, which includes its faces
$d_0(T,\tau),\ldots,d_n(T,\tau)$ and degeneracies
$s_0(T,\tau),\dots,s_n(T,\tau)$, and so-on.
The $n{+}1$ ``vertices'' are (generated by) the single-attribute data tables $(T_0, \tau_0),\ldots,
(T_n,\tau_n)$ obtained by applying sequences of $n$ face maps. 
For $m\leq n$, the $m$-simplices within $(T,\tau)$ are (generated by) the
data tables obtained by applying sequences of face maps and
degeneracy maps until the result has $m{+}1$ attributes.
The picture of ``two simplices that share a boundary component'' is realized in
$\sS{X}'$ as a pair of data tables $(T_{01},\tau_{01})$ and $(T_{02},\tau_{02})$  
and attribute inclusions $\iota_{01}:T_0 \to T_{01}$ and $\iota_{02}:T_0 \to
T_{02}$ such that there is a data table $(T_0, \tau_0)$ with
$\red_{\iota_{01}}\tau_{01} = \tau_0 = \red_{\iota_{02}}\tau_{02}$.
If $\len(T_{01})=\len(T_{02})=2$ and $\len(T_0)=1$, then this information generates a $2$-horn.
A completion of the $2$-horn to a $2$-simplex would be (generated by) a data table
$(T_{012},\tau_{012})$ that has $(T_{01},\tau_{01})$ and $(T_{02},\tau_{02})$
as two of its three faces.    Depending on the available simplices in
$\sS{X}'$, it may or may not be possible to find such $(T_{012},\tau_{012})$.
\label{rmk:datasimp}
\end{rmk}

\begin{defn}[Kan Condition]
A simplicial set $X$ is said to satisfy the \emph{Kan 
condition} iff any map from a horn $\Lambda^n_k \to X$ extends to a
compatible map from the simplex, $\Lambda^n_k \into \Delta^n \to X$.
\end{defn}
The Kan condition means that the simplicial set is closed under
simplicial deformation, so it has a well-defined homotopy group.
The Kan condition is not specific to data complexes;
it is a definition for general simplicial sets, and gives the appropriate
notion of \emph{fibrant} for many model categories.   For our purposes, we require a slight variation on the Kan condition to provide an adequate notion of \emph{fibrant} data contexts, which we now develop as Defn~\ref{def:fibrant}.

We define the space of joins for two data
tables with designated attribute inclusions.
\begin{defn}[Space of Joins]
Suppose attribute lists $T_{0}, T_{01}, T_{02} \in \sS{A}$ are equipped with attribute
inclusions $\iota_{01}:T_{0} \to T_{01}$ and $\iota_{02}:T_{0} \to T_{02}$, and
let $T_{012}$ denote the attribute list $\Merge(T_{01},T_{02},\iota_{01}\sim\iota_{02})$
as in Defn~\ref{defn:merge}, which is equipped with inclusions $\iota'_{01}:T_{01}\into
T_{012}$ and $\iota'_{02}:T_{02}\into T_{012}$.
For any data tables $(T_{01}, \tau_{01}), (T_{02}, \tau_{02}) \in \sS{X}$, let 
\begin{equation}
\JOINS( \tau_{01}, \tau_{02}, \iota_{01}\sim\iota_{02}) \defeq 
\{ (T_{012}, \tau_{012}) \in \sS{X} \,:\, \red_{\iota'_{01}}\tau_{012} = \tau_{01},\
\red_{\iota'_{02}}\tau_{012} = \tau_{02}\}.\end{equation}
\label{def:joinspace}
Note that ``$\JOINS(\tau_{01}, \tau_{02}, \iota_{01}\sim\iota_{02}) \neq \emptyset$'' requires 
$\red_{\iota'_{01}\circ \iota_{01}}\tau_{01} = \red_{\iota'_{01}\circ\iota_{01}}\tau_{02}$, as 
$\red_{T_0} \tau_{012}$ must be well-defined.
\end{defn}
Similarly to Defn~\ref{defn:merge}, we write this set as $\JOINS(\tau_{01},\tau_{02},T_0)$
for notational convenience when the attribute inclusions are understood from
context.

Note!
This is \emph{not} the same notion of ``join'' that one sees in
traditional topology, or in categorical references such as
\cite{Ehlers1999,Rijke2017} and 
\url{https://ncatlab.org/nlab/show/join+of+simplicial+sets}.
It is not yet clear whether there is a useful relationship to joins in ergodic
theory \cite{Glasner2003}.
We choose the term ``join'' to mimic the terminology in database engineering
discussed in Section~\ref{sec:Intro}.  Definition~\ref{def:joinspace} reminds one of couplings from statistics, as in
\cite{Baxter1994}; however, the generality here allows repetition and distinct measures and overlaps.

\begin{defn}[Join Conditions]
A data subcomplex  $p:\sS{X}'\to\sS{A}'$ of an ambient $p:\sS{X}\to\sS{A}$ is said to satisfy the
\emph{weak join condition} iff, for any $(T_{01}, \tau_{01})$ and $(T_{02}, \tau_{02}) \in
\sS{X}'$ with attribute inclusions $T_0 \into T_{01}$ and $T_0
\into T_{02}$ and $\red_{T_0}\tau_{01} = \red_{T_0}\tau_{02}$, then $\JOINS(\tau_{01}, \tau_{02}, T_0) \cap
\sS{X}'$ is nonempty.
It satisfies the \emph{strong join condition} iff $\JOINS(\tau_{01},\tau_{02},T_0) \subset \sS{X}'$.
\label{def:joincond}
\end{defn}
The weak join condition means that the simplicial set admits some database JOIN
operation between any well-aligned pair of data tables. The strong join
condition requires that \emph{all possible} joins exist in $\sS{X}'$.

The trivial join (when $T_0 = [ ]$) is of particular interest as it provides a
generalization of independent products of measures. 
\begin{defn}[Admission of Trivial Joins]
A data subcomplex $p:\sS{X}' \to \sS{A}'$ of an ambient $p:\sS{X}\to\sS{A}$ is said to 
\emph{admit trivial joins} iff, for every $(T_1, \tau_1),
(T_2, \tau_2) \in \sS{X}'$ with $\tot \tau_1 = \tot
\tau_2 = M$, there exists some $\tau_{12} \in T_1 \oplus T_2$ such that
$\red_{T_1} \tau_{12} = \tau_1$ and $\red_{T_2}\tau_{12} = \tau_2$.
\end{defn}
\begin{defn}[Closure under Independent Products]
A data subcomplex $p:\sS{X}' \to \sS{A}'$ of an ambient $p:\sS{X}\to\sS{A}$ is said to be
\emph{closed under independent products} iff, for every $(T_1, \tau_1),
(T_2, \tau_2) \in \sS{X}'$ and $\tot \tau_1 = \tot
\tau_2 = M$, we have $(T_1 \oplus T_2, \frac{\tau_1 \tau_2}{M}) \in \sS{X}'$.
\end{defn}
\begin{rmk}
The independent product is an example of a trivial join.
If a data subcomplex is closed under independent products, then it also includes all IID measures built from its various data tables; this property is important for applications to statistics.
\end{rmk}

\begin{lemma}
If a data subcomplex $p:\sS{X}'\to\sS{A}'$ of an ambient $p:\sS{X}\to\sS{A}$
satisfies the strong join condition, then $p:\sS{X}'\to\sS{A}'$ is closed under permutations.
\label{lem:strongperm}
\end{lemma}
\begin{proof}
Because $A$ is finite, it suffices to prove that $\sS{X}'$ is closed under
permutations that are swaps (that is, transpositions or 2-cycles). Moreover, it
suffices to consider only swaps of adjacent entries, as any swap $i
\leftrightarrow j$ can be written by migration of $j$ past $i$, then $i$ to the
original position of $j$. 

Suppose that $(T,\tau) \in \sS{X}'$ with $T = [a_0, \ldots, a_i,
a_{j}, \ldots, a_n]$ with $j=i+1$.  Then, consider $(T_{01},\tau_{01}) = d_i(T,\tau)$
and $(T_{02},\tau_{02}) = d_i(T,\tau)$, and $(T_{0},\tau_{0}) =
d_jd_i(T,\tau)$.  By construction, there are well-defined attribute inclusions $ T_0
\into T_{01}$ and $T_{0} \into T_{01}$ that satisfy $\red_{T_0}\tau_{01} =
d_jd_i\tau = d_{j-1}d_i\tau = \red_{T_0}\tau_{02}$.
Note that $T= \Merge(T_{01},T_{02},T_0)$, and $\tau \in
\JOINS(\tau_{01},\tau_{02},\tau)$.
Consider the attribute list $\tilde{T} = [a_0, \ldots, a_j, a_i, \ldots,
a_n]$ obtained by swapping the adjacent entries $a_i$ and $a_j$, and let $\tilde{\tau}$ denote the correspondingly permuted measure obtained from $\tau$.
Note that $\tilde{T} = \Merge(T_{02},T_{01},T_0)$, and 
$\tilde{\tau} \in \JOINS(\tau_{02},\tau_{01},T_0)$.  By the strong join
condition, $\tilde{\tau} \in \sS{X}'$.
\end{proof}

\begin{thm}[Fundamental theorem of data complexes]
For any data subcomplex $p:\sS{X}'\to\sS{A}'$ of an ambient $p:\sS{X} \to \sS{A}$.
\begin{enumerate}
\item If $\sS{X}'$ satisfies the strong join condition, then $\sS{X}'$ admits trivial joins and $\sS{X}'$ satisfies the Kan condition as a simplicial set.
\item If $\sS{X}'$ admits trivial joins and $\sS{X}'$ satisfies the Kan condition, then $\sS{X}'$ satisfies the weak join condition.
\end{enumerate}
\label{thm:ft}
\end{thm}

(We would not be surprised if the Kan condition and the strong join condition
are equivalent, under some reasonable assumptions, but we have not pursued that
claim.)

\begin{proof}[Proof of (2).]
Suppose that $\sS{X}'$ satisfies the Kan condition and admits
trivial joins.  Admission of trivial joins provides the
weak join condition in the case $T_0 = [ ]$.  Suppose that $(T_{01},\tau_{01})$ and
$(T_{02},\tau_{02})$ are elements of $\sS{X}'$.  Suppose that there are inclusions
$\iota_{01}:T_0 \to T_{01}$ and $\iota_{02}:T_0 \to T_{02}$ for some $T_0$, and 
suppose that $\red_{T_0} \tau_{01} = \red_{T_0} \tau_{02}$.
Each of $(T_{01},\tau_{01})$ and $(T_{02},\tau_{02})$ and $(T_{0}, \tau_{0})$ provides 
all faces of all lower dimensions.   We prove the existence of 
$(T_{012}, \tau_{012}) \in \JOINS(\tau_{01}, \tau_{02}, T_0)$ by induction on the
dimension.
For simplicity, we use the language of
simplicial sets, instead of the language of measures. Recall that a ``vertex''
is a data table obtained by marginalizing to a single attribute, and an
$n$-simplex is a data table obtained by marginalizing to $n{+}1$ attributes,
as in Remark~\ref{rmk:datasimp}.
Fix a preferred vertex
$k$ in $(T_0, \tau_0)$. For any vertex $i$ in $(T_{01}, \tau_{01})$ and any vertex
$j$ in $(T_{02}, \tau_{02})$, the 1-simplex $[i, k]$ and $[k,j]$ exist \emph{a priori}
(up to notational ordering). This is an example of a horn $\Lambda^2_k$.
Therefore, by the Kan condition, the 2-simplex $[i,k,j]$ exists in $\sS{X}'$.  Hence, every
2-face including $k$ and vertices in $(T_{01},\tau_{01})$ or
$(T_{02},\tau_{02})$ exists in $\sS{X}'$.
Assume for induction that every $n$-face containing vertex $k$
exists.  Any $n$ of those $n$-faces form a horn $\Lambda^{n+1}_k$, so their
$(n{+}1)$-face exists in $\sS{X}'$.  So, every $(n{+}1)$-face containing vertex
$k$ exists in $\sS{X}'$.  Therefore, there is a Data Table $(T_{012},\tau_{012})$ that
involves all vertices in $(T_{01},\tau_{01})$ and $(T_{02},\tau_{02})$.
\end{proof}

\begin{proof}[Proof of (1).]
We prove part (1) under the notable assumption that $\Val(a)$ is a compact
metric space for all $a \in A$.   Hence, for any attribute
list $T$, the space of measures $\Msr(T)$ includes a uniform\footnote{That is,
$\kappa_T(B_r(x))$ depends only on $r$, for metric balls $B_r(x)$ of
sufficiently small radius.} probability measure $\kappa_T$. 

Suppose that $\sS{X}'$ satisfies the strong join condition.  The case
$T_0 = []$ implies admission of trivial joins.

In this proof,  we assume that $k=0$ is the common vertex in a horn
$\Lambda^n_k$, but that
is only for notational simplicity; the proof certainly applies for any other
specified vertex $k$, by appropriate re-ordering.
Consider data tables giving a horn $\Lambda^n_0$.
These data tables are of the form $(T_{\widehat{m}}, \tau_{\widehat{m}})$ for
$1 \leq m \leq n$, where $T_{\widehat{m}}= [a_0,\ldots, a_{m-1}, a_{m+1}, \ldots,
a_n]$.  Let $T=[a_0,\ldots,a_n]$.  As a horn, these data tables are
well-aligned; that is, they match on all corresponding faces according to $d_i
\tau_{\widehat{j}}=d_{j-1}\tau_{\widehat{i}}$ for $i<j$ as noted after
Definition~\ref{def:horn}.
In particular, all these data tables share the same total mass, $M$.
To establish the Kan condition, we construct a compatible $n$-simplex; that is, 
a data table $(T,\tau)$ such that $d_m(T,\tau) = (T_{\widehat{m}},\tau_{\widehat{m}})$ for $1 \leq m\leq n$.

A brief outline of the argument:
First, we construct a data table $(T,\tau^{(n-1)})$ such that
$d_{n-1}\tau^{(n-1)} = \tau_{\widehat{n-1}}$ and $d_{n}\tau^{(n-1)} =
\tau_{\widehat{n}}$.  The measure $\tau^{(n-1)}$ is built from a parameterized
family of partial measures on
$\Val(T_{n-1}\oplus T_{n})$ by recursively bifurcating the parameter
set $\Val(T_{0\cdots(n-2)})$ into dyadic sets, which allows $\tau^{(n-1)}$ to be defined via
countable disjoint unions.  This data table $\tau^{(n-1)}$ serves as the base
case for an inductive argument for a sequence of partial solutions
$\tau^{(n-1)}, \ldots, \tau^{(m)}, \ldots, \tau^{(1)}$ such that $\tau^{(m)}$ has
the desired faces $d_{m}$ through $d_{n}$.  Finally,
$(T,\tau^{(1)})$ is the desired $n$-simplex.  Let us proceed.

For $K=\Val(T_{0\cdots(n-2)})$, a compact set, consider the measures
\begin{equation} 
\begin{cases}
\mu_{K,n-1}:&U_{n{-}1} \mapsto \tau_{\widehat{n}}(K \ptimes U_{n{-}1})\ \text{in}\
\Msr(T_{n-1}),\\
\mu_{K,n}:&U_{n} \mapsto \tau_{\widehat{n-1}}(K \ptimes U_{n})\ \text{in}\
\Msr(T_{n}).  
\end{cases}
\label{eqn:startK}
\end{equation}
Consider any trivial join $t_K \in \JOINS(\mu_{K,n-1}, \mu_{K,n}, [ ]) \subset
\Msr(T_{n-1} \oplus T_{n}) = \Msr(\Merge(T_{n-1}, T_{n}, []))$; that is,
$\red_{T_{n-1}} t_K(U_{n-1}) = \tau_{\widehat{n-1}}(K \ptimes U_{n-1})$ and 
$\red_{T_{n}} t_K(U_{n}) = \tau_{\widehat{n-1}}(K \ptimes U_{n})$.
Note that
\begin{equation}
\tot t_K 
= d_n \tau_{\widehat{n-1}}(K) 
= d_n \tau_{\widehat{n}}(K)
= \tau_{0\cdots(n-2)}(K)=M.
\end{equation}
Fix any open $W \subset K$ such\footnote{Of course the value of $\frac12$ is
not special, but aesthetic.  Any $0< \kappa(W) < \kappa(K)$ will do.} that $\kappa_{T_{0\cdots(n-2)}}(W) =
\frac12\kappa_{T_{0\cdots(n-2)}}(K)$.  Note that the measures 
\begin{equation} 
\begin{cases}
\mu_{W,n-1}:&U_{n{-}1} \mapsto \tau_{\widehat{n}}(W \ptimes U_{n{-}1})\ \text{in}\
\Msr(T_{n-1}),\\
\mu_{W,n}:&U_{n} \mapsto \tau_{\widehat{n-1}}(W \ptimes U_{n})\ \text{in}\
\Msr(T_{n})
\end{cases}
\label{eqn:startW}
\end{equation}
can be joined to provide  
$t_W \in \JOINS(\mu_{W,n-1}, \mu_{W,n}, [ ]) \subset
\Msr(T_{n-1} \oplus T_{n}) = \Msr(\Merge(T_{n-1}, T_{n}, []))$. That is,
$\red_{T_{n-1}} t_W(U_{n-1}) = \tau_{\widehat{n-1}}(W \ptimes U_{n-1})$ and 
$\red_{T_{n}} t_W(U_{n}) = \tau_{\widehat{n-1}}(W \ptimes U_{n})$.
Note that \begin{equation}\tot t_W = d_{n-1} \tau_{\widehat{n-1}}(W) = d_{n-1} \tau_{\widehat{n}}(W)
=\tau_{0\cdots(n-2)}(\Val(T_{0\cdots(n-2)})).
\end{equation}
Further, by their definitions via trivial joins from $W \subset K$, one can
choose $t_W$ to guarantee that  
$t_W(U_{n-1}\times U_{n}) \leq t_K(U_{n-1}\times U_{n})$
for all Borel sets $U_{n-1}\times U_{n} \subset \Val(T_{n-1}\oplus T_{n})$.
In particular, $\tot t_W \leq \tot t_K$.
Likewise, for the closed set $K-W$, define $t_{K-W} \defeq t_{K} - t_{W}$,
which is also a measure in $\Msr(T_{n-1}\oplus T_{n})$ by construction.
Note that both the closure $\bar{W}$ and the complement $K-W$ are closed in
$K$, therefore both are compact.  Replacing $K$ with $\bar{W}$ or $K-W$ in
\eqref{eqn:startK} means that we can establish measures $\{
t_{W_\lambda}\}_{\lambda \in \Lambda}$ for a countable bifurcating collection
$\{W_\lambda\}_{\lambda\in \Lambda}$ of open sets; any
measurable set in $\Val(T_{0\cdots(n-2)})$ can be $\kappa$-almost covered
by disjoint sets the collection.  By analogy, we refer to the
$\{W_\lambda\}_{\lambda\in\Lambda}$ as a dyadic collection.

Given a such a countable collection $\{t_{W_\lambda} \}_{\lambda \in \Lambda}
\subset \Msr(T_{n-1}\oplus T_{n})$, 
define a measure
$\tau^{(n-1)} \in \Msr(T)$ on Borel sets $U_0 \ptimes \cdots \ptimes U_n$
by disjoint $\sigma$-additivity,
\begin{equation}
\tau^{(n-1)}(U_0 \ptimes \cdots \ptimes U_n) \defeq
\sum_{\substack{\text{disjoint}\\W_\lambda \subset
U_0\ptimes \cdots \ptimes U_{n-2}}} t_{W_\lambda}(U_{n-1}\times U_{n}).
\end{equation}
By construction,
$d_{n}\tau^{(n-1)} = \tau_{\widehat{n}} \in \sS{X}'$ and 
$d_{n-1}\tau^{(n-1)} = \tau_{\widehat{n-1}} \in \sS{X}'$.  Therefore, by the
strong join condition, $(T,\tau^{(n-1)}) \in \sS{X}'$.
The data table $(T,\tau^{(n-1)})$ provides the base case for induction on faces.

Assume for induction that for some $m$ satisfying $1 < m \leq n-1$ there exists
a data table  $(T,\tau^{(m)}) \in \sS{X}'$ such that 
$d_{k}\tau^{(m)} = \red_{T_{\widehat{k}}}\tau^{(m)}  = \tau_{\hat{k}} \in
\sS{X}'$ for all $m \leq k \leq n$.
Denote the ``error'' of the $d_{m-1}$ face as 
\begin{equation}
\varepsilon_{m-1} \defeq \left( d_{m-1}\tau^{(m)} -
\tau_{\widehat{m-1}}\right).
\end{equation}
The error $\varepsilon_{m-1}$ is a signed measure---not a measure---on
$\Val(T_{\widehat{m-1}})$, but the face operation of marginalization is still sensible.
Then for $m \leq k \leq n-1$,
\begin{equation}
d_k \varepsilon_{m-1} 
=
d_k(d_{m-1}\tau^{(m)} - \tau_{\widehat{m-1}})
=
d_{m-1}d_{k+1}\tau^{(m)} - d_k \tau_{\widehat{m-1}}
=
d_{m-1}\tau_{\widehat{k+1}} - d_{m-1} \tau_{\widehat{k+1}}=0.
\label{eqn:eps0}
\end{equation}

Also, for application below, consider the pre-measure $f$ on Borel sets
$U_{m-1} \subset \Val(T_{m-1})$ as defined by
\begin{equation}
f(U_{m-1}) = \inf\left\{ 
\frac{\tau^{(m)}(W \ptimes U_{m-1} \ptimes Z)}{\varepsilon_{m-1}(W
\ptimes Z)}\   
\text{for Borel $W \ptimes Z \subset \Val(T_{\widehat{n-1}})$ st
$\varepsilon_{m-1}(W\ptimes Z)>0$}\right\}.
\label{fprem}
\end{equation}
Observe the inequality 
\begin{equation}
f(\Val(T_{m-1})) \geq 1,
\label{INEQ0}
\end{equation}
which follows because
for all Borel sets $W \ptimes Z \subset \Val(T_{\widehat{n-1}})$ satisfying 
$\varepsilon_{m-1}(W\ptimes Z)>0$, we have
\begin{equation}
\begin{split}
\frac{\tau^{(m)}(W \ptimes \Val(T_{m-1}) \ptimes Z)}{\varepsilon_{m-1}(W
\ptimes Z)} = 
\frac{\tau_{\widehat{m-1}}(W \ptimes Z) + \varepsilon_{m-1}(W \ptimes Z)}{\varepsilon_{m-1}(W
\ptimes Z)}= 
1+ \frac{\tau_{\widehat{m-1}}}{\varepsilon_{m-1}}(W \ptimes Z) 
\geq 1.\\
\end{split}
\label{INEQ1}
\end{equation}
Let $\rho_{m-1} \in \Msr(T_{m-1})$ be a probability measure satisfying the
condition
\begin{equation}
\rho_{m-1}(U_{m-1})
\leq
f(U_{m-1}) 
\label{INEQ2}
\end{equation}
for all Borel $U_{m-1}\subseteq \Val(T_{m-1})$.
Such probability measures are guaranteed to exist by \eqref{INEQ0}.

Then, define
for any\footnote{Not necessarily meeting the $\varepsilon_{m-1}>0$ condition above.} 
Borel $W \ptimes U_{m-1} \ptimes Z \subseteq \Val(T)$,
\begin{equation}
\tau^{(m-1)}(W \ptimes U_{m-1} \ptimes Z) \defeq 
\tau^{(m)}(W \ptimes U_{m-1} \ptimes Z) 
- \rho_{m-1}(U_{m-1}) \cdot \varepsilon_{m-1}(W \ptimes Z),
\label{eqn:deftm1}
\end{equation}
and extend by additivity.
By construction, $\tau^{(m-1)}$ is additive and zero-null. Non-negativity
follows from \eqref{INEQ2} and the definition of $f$ in \eqref{fprem}; therefore, $\tau^{(m-1)}$ is a measure on
$\Val(T)$.
Moreover, $\tau^{(m-1)}$ satisfies the desired marginalizations, shown here:
\begin{equation}
\begin{split}
d_{m-1} \tau^{(m-1)}(W \ptimes Z)
&= \tau^{(m)}(W \ptimes Z) - \rho_{m-1}(\Val(T_{m-1})) \cdot \varepsilon_{m-1}(W \ptimes Z)\\
&= \tau^{(m)}(W \ptimes Z) - 1 \cdot \varepsilon_{m-1}(W \ptimes Z)\\
&=  \tau_{\widehat{m-1}}(W \ptimes Z).
\end{split}
\end{equation}
And, for $m \leq k \leq n$, the properties \eqref{eqn:eps0} apply to give
\begin{equation}
d_{k} \tau^{(m-1)} 
=
d_{k}\left(\tau^{(m)} 
- \rho_{m-1} \cdot \varepsilon_{m-1}\right)
=
d_{k}\tau^{(m)} 
- \rho_{m-1}\cdot0
=
d_{k}\tau^{(m)}
=
\tau_{\widehat{k}}.
\end{equation}
Therefore, $(T,\tau^{(m-1)})$ is a data table that has the desired faces
$d_{m-1}$ through $d_{n}$.  The inductive step is established.
The ultimate data table $(T,\tau^{(1)})$ provides the $n$-simplex $\Delta^n$
completing $\Lambda^n_0$.
\end{proof}

\begin{rmk}[Freedom]
The flexibility in choosing $(T,\tau)$ arises from an initial parametric choice
of joined measures $\{t_x\}_{x \in \Val(T_{0\cdots(n-2)})} \subset
\Msr(T_{n-1}\oplus T_n)$ and a finite set of probability measures $\rho_{n-2}$,
$\rho_{n-3}$, \ldots, $\rho_{1}$.
\end{rmk}

\subsection{Simplicial Homotopy of Data Complexes}
\label{sec:SHDC}
\begin{defn}[Fibrant]
A data complex $p:\sS{X}' \to \sS{A}'$ satisfying the strong join condition is called fibrant.
\label{def:fibrant}
\end{defn}
See Appendix~\ref{sec:sset} for a categorical version of this definition.
The entire \emph{raison d'\^etre} of fibrant objects is that they admit
homotopy, as proven by \cite{Kan1958} and \cite{Quillen1967}, which allows
obstruction theory to be studied in direct analogy to Steenrod.
In the category of simplicial sets, the term \emph{fibrant} refers only to the
Kan extension condition.  Our practical desire to use joins as a
weak-equivalence compels us to require the strong join condition.  By
Theorem~\ref{thm:ft}(1) the traditional definition and all of its consequences
are implied.
\begin{cor}
Suppose a data subcomplex $\sS{X}'$ of an ambient $\sS{X}$ is fibrant, and fix a
basepoint data table $(T_0,\tau_0)$.  The homotopy group
$\pi_n(\sS{X}', \tau_0)$ is well-defined for all $n$, and satisfies the typical
properties of homotopy categories over model categories.
\end{cor}

\begin{thm}
For any attribute set $A$ and value spaces $\Val(~)$, the ambient data complex $p:\sS{X}\to\sS{A}$ is fibrant.
\label{thm:Xfibrant}
\end{thm}
\begin{proof}
The very definition of an ambient $\sS{X}$ is that it includes \emph{all} finite measures over the relevant metric spaces, so it includes the set $\JOINS( )$ in particular.
\end{proof}
We now want to explore how a data subcomplex $p:\sS{S}\to\sS{B}$ of an ambient $p:\sS{X} \to \sS{A}$ interacts with 
any other attribute list $T \in \sS{A}$.  The following sets are of interest.
\begin{defn}
Let $\sS{S}|_T \defeq \{(S,\sigma) \in \sS{S} \,:\, \exists\, \iota: S \into T \}$,
the set of data tables in the data subcomplex that are detected by $T$.
Let $\sS{B}|_T \defeq \{ S \in \sS{B} \,:\, \exists\, \iota:S \into T\} = p(\sS{S}|_T)$, 
the set of attribute lists in the subcomplex that are detected by $T$.
\end{defn}

A data subcomplex $p:\sS{S}\to\sS{B}$ may not be fibrant, so we define
a convenient fibrant space that contains it.
The notation $\sS{F}^0$ is meant to be suggestive; in Section \ref{sec:FO}, a larger filtration of simplicial sets
is created (Definition \ref{defn:approx}) by turning the equality in the definition below into an inequality involving Wasserstein distance.
\begin{defn}[Complex of Perfect Joins]
For any data subcomplex $\sS{X}'$ of an ambient $\sS{X}$, let
$\sS{F}^0$ denote the subset of $\sS{X}$ defined by
\begin{quotation}
$(T,\tau) \in \sS{F}^0$ if and only if $\forall\, a \in T$, $\exists\,
(S,\sigma) \in \sS{S}|_T$ such that $a \in S$ and $\red_S \tau =
\sigma$.
\end{quotation}
Note: the quantifier ``$\forall a \in T$'' refers to each entry in the attribute list, which 
means repeated entries must have corresponding measures.
\end{defn}

The definition of $\sS{F}^0$ is a convenient way to say ``consider everything
that can be generated from $\sS{S}$ using $\JOINS( )$,'' as
justified by the following lemma. Similarly, the upcoming Definition
\ref{defn:approx} of $\sS{F}^t$ gives a convenient way of saying ``consider everything
that can be approximated to an acceptable level of uncertainty from $\sS{S}$ using $\JOINS( )$.''

\begin{lemma}
$(T,\tau) \in \sS{F}^0$ if and only if there is a sequence $(T_0, \tau_0), (T_1,
\tau_1), \ldots, (T_k, \tau_k)$ such that 
\begin{itemize}
\item $(T_0, \tau_0) = (S_0, \sigma_0) \in \sS{S}$, and 
\item $\forall i = 1, \ldots, k,\ (T_i, \tau_i) \in \JOINS(\tau_{i-1},\sigma_i,
T_{i-1} \cap S_i)$ for some $(S_i, \sigma_i) \in \sS{S}$, and
\item $(T_k, \tau_k) = (T,\tau)$.
\end{itemize}
\end{lemma}

\begin{proof}
Suppose $(T,\tau) \in \sS{F}^0$.  Let $a_0 \in T$ denote the first attribute of
$T$.  By the definition of $\sS{F}^0$, there exists $(S_0,
\sigma_0) \in \sS{S}$ with an attribute inclusion $\iota_0:S_0 \into T$ such that
$\red_{\iota_0}\tau = \sigma_0$ and such that $a_0$ is in the image of
$\iota_0$.  Let $(T_0, \tau_0) = (S_0,\sigma_0)$.  By reducing $T_0$ if
necessary, we may ensure that $\iota_0(T_0)$ is contiguous within $T$.
If $T_0 = T$, then the sequence is complete.

Otherwise, there exists some first attribute $a_1$ in $T/\iota_0$.
By the definition of $\sS{F}^0$, there exists $(S_1, \sigma_1) \in \sS{S}$ with
an attribute inclusion $\iota_1:S_1 \into T$ such that $\red_{\iota_1}\tau =
\sigma_1$ and such that $a_1$ is in the image of $\iota_1$.
By reducing $S_1$ if necessary, we may ensure that $\iota_1(S_1)$ is contiguous
within $T$, and that $T_0 \cap S_1$ is also contiguous.  With these reductions, 
the orderings are consistent such that  
$T_1 \defeq \Merge(T_0,S_1, T_0 \cap S_1 )$ is equipped with a list inclusion
$T_1 \into T$.  Because $\tau$ is given, let $\tau_1 =
\red_{T_1}\tau$, which by construction is an element of
$\JOINS(\tau_0,\sigma_1, T_0\cap S_1)$.
Repeat this process until all elements $a_i$ of $T$ are in the image of some
inclusion $T_i \into T$.

For the converse, note that each $a \in T$ is included in some $S_i$, which is
sufficient.
\end{proof}

\begin{cor}
$\sS{F}^0$ includes all independent products formed from data tables in $\sS{S}$.
\end{cor}

\begin{lemma}
For any data subcomplex $\sS{S}$ of an ambient $\sS{X}$, the complex of perfect joins
$\sS{F}^0$ is fibrant. 
\label{lem:F0}
\end{lemma}
\begin{proof}
The data subcomplex $\sS{S}$ is closed under face maps and degeneracy
maps, so application of those maps to all $(S,\sigma)$ in the definition shows
that $\sS{F}^0$ is closed under the face maps and degeneracy maps as well.
To verify that $\sS{F}^0$ is fibrant, suppose that $(T_{012}, \tau_{012}) \in \sS{X}$ is any join of $(T_{01},\tau_{01})$ and $(T_{02},\tau_{02})$ in $\sS{F}^0$.
Because every $a \in T_{012}$ appears in $T_{01}$ or $T_{02}$, the existence of
$(S,\sigma) \in \sS{S}$ in inherited from $(T_{01},\tau_{01})$ and
$(T_{02},\tau_{02})$.
\end{proof}

We conclude this section by tying simplicial homotopy theory to Problem \ref{GPmerge2}.
\begin{lemma}
Suppose $\sS{X}'$ is a fibrant data subcomplex of an ambient $\sS{X}$.
A basepoint-preserving simplicial map $f:\partial \Delta^{n} \to \sS{X}'$
defines a class in $\alpha(f) \in \pi_{n-1}(\sS{X}')$.  Moreover,
$\alpha(f)=e$ if and only if $f$ admits an extension $f^+:\Delta^{n} \to \sS{X}'$.
\end{lemma}
\begin{proof}
The first claim reduces to Lemma 9.6 in \cite{Friedman2008}.  The second claim
reduces to Lemma 7.4 in \cite{Seriesa}.  Our definition of fibrant implies path-connectedness, so 
a spanning tree can be used for locality such as in \cite{Kan1958}.
\end{proof}

\begin{cor}
Suppose that $p:\sS{S}\to\sS{B}$ is a data subcomplex of an ambient $p:\sS{X}\to\sS{A}$ such that $\sS{B}_{n-1}  = \sS{A}_{n-1}$ for some $n\geq 1$.
Fix a simplicial section $\sigma:\sS{B}_{n-1} \to \sS{S}_{n-1}$.  The following are equivalent (omitting basepoints for brevity).
\begin{enumerate}
\item 
For every composition 
\[\partial \Delta^{n} \overset{c}{\to} \sS{B}_{n-1} \overset{\sigma}{\to} \sS{S}_{n-1} \overset{\iota}{\to} \sS{F}^0,\]
we have $\alpha(\iota \circ  \sigma \circ c) = e \in \pi_{n-1}(\sS{F}^0)$. 
\item $\sigma$ admits an extension of the form $\sigma^+:\sS{A}_{n} \to \sS{F}^0_{n}$.
\end{enumerate}
\end{cor}
\begin{proof}
Because $\sS{A}_{n-1} = \sS{B}_{n-1}$, the boundary of every $n$-simplex in $\sS{A}$ appears in $\sS{B}$.
Apply the previous lemma for each $f= \iota \circ \sigma \circ c$ as a map $f:\partial \Delta^{n} \to \sS{X}'$ for $\sS{X}' = \sS{F}^0$. 
\end{proof}
This corollary is revisited as Lemma~\ref{thm:cocycle0}. The corollary fails when no such extension can be found.  Then, the question remains: how to measure the failure of this corollary?
That measurement is the purpose of filtered obstruction theory.

\section{Filtrations and Obstructions}
\label{sec:FO}
This section concludes the theoretical framework outlined in Section~\ref{sec:IO}.
Together, obstructions and filtrations allow us to detect when merging is possible; if merging appears obstructed, we can determine whether merging can be achieved by reverting a previous merge or by altering some of the data tables.
Section \ref{sec:FDS} introduces a filtration from a data subcomplex $\sS{S}$ to its ambient $\sS{X}$ using the Wasserstein distance.
Each level of the filtration is fibrant, which allows one to define an obstruction cocycle (Section \ref{sec:PoT}) at each level of the filtration.
Eventually, for a high enough level in the filtration, the obstruction cocycle becomes trivial, so the importance of the obstruction cocycle can be measured using topological persistence.
This statement is formalized in Theorem \ref{thm:tri}, which can be seen as the main payoff of our theoretical development in terms of database engineering. As promised in the introduction, the theory of data complexes does not just mathematize the notion of table merging; rather, it provides further powerful operations when traditional merging is impossible.

\subsection{Filtrations from Data Subcomplexes}
\label{sec:FDS}

A general notion of persistence on simplicial sets appears in \cite{Otter2018}.  In summary,
a fibrant filtration of simplicial sets is a bi-graded collection of sets
$\{ \sS{F}^t_n \}$ for $0 \leq t \leq \infty$ and $n \in \mathbb{N}$ equipped with 
maps $d_i$ and $s_i$ such that
\begin{enumerate}
\item $(\sS{F}^t, d_i, s_i)$ is a simplicial set for each $t$,
\item $\sS{F}^{s}_i \subseteq \sS{F}^{t}_i$ for all $s \leq t$, and
\item $(\sS{F}^t, d_i, s_i)$ is fibrant for each $t$.
\end{enumerate}
The fibrant condition implies that $\pi_n(\sS{F}^t)$ is well-defined for all $t$, and the inclusion maps $\sS{F}^s \into \sS{F}^t$ induce maps on homotopy, $\pi_n(\sS{F}^s) \to
\pi_n(\sS{F}^t)$.

We now define a specific filtration for a data subcomplex that is designed to
meet our application regarding joining data tables.
Recall that $(\Val(a),\rho_a)$ is a Radon space for each attribute $a$.

\begin{defn}[Wasserstein Distance]
For any $a \in A$ with $(\Val(a), \rho_a)$, and $\tau_1,\tau_2 \in \Msr(a)$, let
\begin{equation}
w_a(\tau_1, \tau_2) \defeq \inf\left\{ \int_{\Val([a,a])}
\rho_a(x_1,x_2)\, \mathrm{d}\mu(x_1,x_2) ~:~ \mu
\in \Msr([a,a]),\ \red_{1}\mu = \tau_1,\ \red_{2} \mu = \tau_2\right\}.
\end{equation}
The reductions $\red_{1}$ and $\red_{2}$ refer to the two copies of the
attribute $a$.

For any $T \in \sS{A}$ and $\tau_1,\tau_2 \in \Msr(T)$, let 
\begin{equation}
w_T(\tau_1, \tau_2) \defeq \inf\left\{\int_{\Val(T \oplus T)} \rho_T(x_1,x_2)
\mathrm{d}\mu(x_1,x_2) \,:\, \mu
\in \Msr(T\oplus T),\ \red_{1} \mu = \tau_1,\ \red_{2} \mu = \tau_2\right\}.
\end{equation}
The reductions $\red_{1}$ and $\red_{2}$ refer to the two interwoven copies of the
attribute list $T$.
\label{def:wasserstein}\end{defn}
\begin{rmk}
Recall that $\rho_T(x_1, x_2) = \max_{a \in T}
\rho_a(x_{1,a},x_{2,a})$, the $L^\infty$-metric obtained from the individual attribute metrics.
Also, in the special case that $\tot\tau_1 = \tot \tau_2$, the infimum argument $\mu$ lies
in the space of trivial joins, $\JOINS(\tau_1, \tau_2, [ ])$, so the Wasserstein distance is tied
to our notion of fibrant data complexes.
\end{rmk}

\begin{lemma}
Suppose that $\tau_1, \tau_2 \in T$. 
If $w_T(\tau_1, \tau_2) = t$, then $w_{d_i T}(d_i \tau_1, d_i \tau_2) \leq
t$.
\end{lemma}
\begin{proof}
Let $a$ indicate the $i$th attribute of $T$, and write $T'=d_iT$ with inclusion
$T' \into T$ and quotient inclusion $[a] \into T$.
Then write $(T',\tau'_1) \defeq d_i (T,\tau_1)$ and 
$(T',\tau'_2) \defeq d_i (T,\tau_2)$ and 
For any $\mu \in \Msr(T \oplus T)$ such that $\red_1 \mu = \tau_1$ and
$\red_2 \mu = \tau_2$, let $\mu' \in \Msr(T' \oplus T')$ be the reduction of
$\mu$ obtained by applying both copies of $\red_{T'}=d_i$.  Also, we use the
notational convention $x = (y,z)$ for $x \in \Val(T)$, $y\in \Val(T')$, $z \in
\Val([a])$, so the $L^\infty$ metric gives $\rho_T(x_1,x_2) \geq
\rho_{T'}(y_1,y_2)$.
\begin{equation}
\begin{split}
&
\int_{(x_1, x_2) \in \Val(T \oplus T)} \rho_{T}(x_{1},x_{2})\mathrm{d}\mu(x_1,x_2) \\
&\geq
\int_{((y_1,z_1), (y_2,z_2)) \in \Val(T \oplus T)} \rho_{T'}(y_1, y_2)
\mathrm{d}\mu(
(y_1,z_1),(y_2,z_2)) \\
&=
\int_{(y_1,y_2) \in \Val(T' \oplus T')}\int_{(z_1,z_2) \in \Val([a,a])}
\rho_{T'}(y_1,y_2)\mathrm{d}\mu( (y_1,z_1),(y_2,z_2)) \\
&= \int_{(y_1, y_2) \in \Val(T \oplus T)} \rho_{T'}
(y_{1},y_{2})\mathrm{d}\mu'( y_1,y_2).
\end{split}
\end{equation}
Therefore, the infimum defining $w_{T'}(\tau'_1, \tau'_2)$ cannot be greater than the
infimum defining $w_{T}(\tau_1, \tau_2)$. 
\end{proof}

\begin{lemma}
Suppose that $\tau_1, \tau_2 \in T$. 
If $w_T(\tau_1, \tau_2) = t$, then $w_{s_i T}(s_i \tau_1, s_i \tau_2) \leq
t$.
\end{lemma}
\begin{proof}
Let $a$ indicate the $i$th attribute of $T$.  Let $T^+ \defeq s_i T$, equipped with the degeneracy inclusion $T\into T^+$ and its quotient $[a]\into T^+$.
Then write $(T^+, \tau^+_1) \defeq s_i (T,\tau_1)$
and $(T^+, \tau^+_2) \defeq s_i (T,\tau_2)$.
Now, for any $\mu \in \Msr(T \oplus T)$ such that $\red_1 \mu = \tau_1$ and
$\red_2 \mu = \tau_2$, let $\mu^+ \in \Msr(T^+ \oplus T^+)$ be the degeneracy of
$\mu$ obtained by applying both copies of $s_i$. 
Consider the integral $\int_{(x_1,x_2) \in \Val(T^+ \oplus T^+)}
\rho_{T^+}\mathrm{d}\mu^+$.  Note that the distributional form of the
degeneracy is a delta,
\begin{equation}
\mathrm{d}\mu^+(x_1, x_2) = 
\begin{cases}
\mathrm{d}\mu(y_1,y_2),& \text{if $x_1 = s_i y_1$, $x_2 = s_i y_2$ for some $(y_1,y_2)
\in \Val(T \oplus T)$,}\\
0,& \text{otherwise}.
\end{cases}
\end{equation}
Moreover, if $x_1 = s_i y_1$, $x_2 = s_i y_2$ for some $(y_1,y_2) \in \Val(T
\oplus T)$, then 
$\rho_{T^+}(x_1, x_2) = \rho_T (y_1,y_2)$.
Together, these give
$\int_{(x_1,x_2) \in \Val(T^+ \oplus T^+)}
\rho_{T^+}\mu^+ = \int_{(y_1, y_2) \in \Val(T \oplus T)} \rho_T \mathrm{d}\mu$.
Therefore, the infimum defining $w_{s_i T}(s_i\tau_1, s_i\tau_2)$ cannot be greater
than the infimum defining $w_{T}(\tau_1,\tau_2)$.
\end{proof}

Now we produce a particular fibrant filtration for a data subcomplex.
\begin{defn}[The Complex of Approximate Joins]
\label{defn:approx}
Let $p:\sS{S}\to\sS{B}$ be a data subcomplex of an ambient $p:\sS{X}\to\sS{A}$. 
For any $0 \leq t \leq \infty$, let 
\[
\sS{F}^t \defeq \{ (T,\tau) \in \sS{X} ~:~ \forall a \in T, \exists (S,
\sigma) \in \sS{S},\ [a]\into  S\into T,\ w_{S}(\red_S \tau, \sigma)
\leq t\}.\] 
\end{defn}
Note that the case $t=0$ reproduces the complex of perfect joins,  
$\sS{F}^0$. Note also that $\sS{F}^\infty = \sS{X}$.

\begin{thm}
For each $t \in [0,\infty]$, $\sS{F}^t$ is a fibrant data subcomplex of $\sS{X}$.
\end{thm}
The proof is identical to the proof of Lemma~\ref{lem:F0}, replacing the
equality with an inequality.
\begin{proof}
Recall that the data complex $\sS{S}$ is closed under face maps and degeneracy
maps. Note the face and degeneracy bounds for the Wasserstein distance given
above. Application of those maps to the $(S,\sigma)$ and $(T,\tau)$ in the definition shows
that $\sS{F}^t$ is closed under the face maps and degeneracy maps as well.
Therefore, $\sS{F}^t$ is a data subcomplex.

To verify that $\sS{F}^t$ is fibrant, apply Theorem~\ref{thm:Xfibrant} to
obtain all joins 
$(T_{012}, \tau_{012}) \in \sS{X}$ from 
any $(T_{01},\tau_{01})$ and $(T_{02},\tau_{02})$ in $\sS{F}^t$.
We must show such $\tau_{012}$ lies in $\sS{F}^t$.  Fix $a \in T_{012}$.
Because every $a \in T_{012}$, it appears in $T_{01}$ or $T_{02}$.  For
concreteness, assume $a \in T_{01}$.
There is some  $(S,\sigma) \in \sS{S}$ such that 
$w_S(\red_S \tau_{01}, \sigma) \leq t$.  By the construction of $\tau_{012}$,
we have $\red_{T_{01}}\tau_{012} = \tau_{01}$, so 
$\red_S \tau_{012}  = \red_S \tau_{01}$.
Hence,
$w_S(\red_S \tau_{012}, \sigma) \leq t$. 
\end{proof}

Because $\sS{F}^t$ is fibrant, all of the usual consequences apply in homotopical algebra, such as
\begin{cor}
Fix a data subcomplex $\sS{S}$ of an ambient $\sS{X}$.
For each $t \in [0,\infty]$, and for each $n\geq 0$, the pointed homotopy group $\pi_n(\sS{F}^t,*)$ is well-defined.
Moreover, for $t_1 \leq t_2$, the inclusion of data subcomplexes $\sS{F}^{t_1} \subset \sS{F}^{t_2}$ induces a homomorphism of pointed homotopy groups $\pi_n(\sS{F}^{t_1},*) \to \pi_n(\sS{F}^{t_2},*)$.
\end{cor}

\subsection{Persistent Obstruction Theory for Data Subcomplexes}
\label{sec:PoT}
Because we have established fibrant objects with resulting homotopy and
homology, we are equipped to extend obstruction theory to our application.
Although our category is not classical, the next several results are modeled on the classical work summarized in
Section 6 of \cite{Thurber1997}, Section 34 of \cite{Steenrod1943}, Section 4 of \cite{Lundell1960}, and
\cite{Olum1950}.  
The discussion culminates in Definition~\ref{def:cocycle} and Theorem~\ref{thm:tri}.

\begin{defn}[Obstruction Cocycle]
Let $\sS{S} \subseteq \sS{F}^0 \subset \cdots \sS{F}^t \subset \cdots
\subset \sS{F}^\infty = \sS{X}$ be the filtration of a path-connected data complex.
Fix a dimension $n$ such that $d_i Y \in \sS{B}_{n-1}$ for all faces $d_i$ of all
$Y \in  \sS{A}_{n}$.
Let $\sigma:\sS{B}_{n-1} \to \sS{S}_{n-1}$ be a data section. 
For a fixed basepoint $(T_0,\tau_0) \in \sS{S} \subset \sS{F}^0$,
define \begin{equation}
\xi^t_\sigma \in C^{n}(\sS{A}, R; \pi_{n-1}(\sS{F}^t, (T_0, \tau_0)))
\end{equation}
to be the element of $\pi_{n-1}(\sS{F}^t,
(T_0, \tau_0))$ that is represented by the loop corresponding\footnote{The well-definedness of this loop is implied by our assumption $R=\Zb$.} to
the cycle $\sigma(\partial Y) \in C_{n-1}(\sS{S}) \subset C_{n-1}(\sS{F}^t)$ for any $Y \in \sS{A}_{n}$.  Extend by linearity
for $Y \in C_n(\sS{A}, R)$.
We typically omit the basepoint and ring for brevity, so $\xi^t_\sigma \in
C^{n}(\sS{A}; \pi_{n-1}(\sS{F}^t))$.
\label{def:cocycle}
\end{defn}

\begin{lemma}
Fix $Y \in \sS{A}_{n}$.  
If $\xi^t_\sigma(Y) = e \in \pi_{n-1}(\sS{F}^t)$, then there exists $(Y, \tau)
\in \sS{F}^t_n$ such that the diagram commutes
\begin{center}
\begin{tikzpicture}[node distance=2cm]
\node (F) {$\sS{F}^t$};
\node[left of=F] (dY) {$\partial Y$};
\node[left of=dY] (dn) {$\partial \Delta^{n}$};
\node[below of=dY] (Y) {$Y$};
\node[below of=dn] (n) {$\Delta^{n}$};
\draw[->] (dn) to node[below] {$\cong$} (dY);
\draw[->] (n) to node[above] {$\cong$} (Y);
\draw[->] (dY) to node[below] {$\sigma$} (F);
\draw[->] (Y) to node[below] {$\tau$} (F);
\draw[->] (dY) to node[left] {$\iota$} (Y);
\draw[->] (dn) to node[left] {$\iota$} (n);
\end{tikzpicture}
\end{center}
\label{thm:cocycle0}
\end{lemma}

\begin{lemma}
The cochain $\xi^t_\sigma$ is a cocycle.
So, it defines a cohomology class $[\xi^t_\sigma] \in H^{n}(\sS{A}, R; \pi_{n-1}(\sS{F}^t, (T_0, \tau_0)))$.
\end{lemma}
\begin{proof}
For any $X \in\sS{A}_{n+1}$,  we have $\delta \xi^t_\sigma(X) = \xi^t_\sigma(\partial
X)$, but then the trivial cycle $0 = \partial (\partial X)  \in C_{n}(\sS{A})$
represents the trivial class $e \in \pi_{n-1}(\sS{F}^0)$.
\end{proof}

\begin{rmk}\label{rmk:meaning}
Obstructions in dimension $n-1=1$ detect loops in $\sS{F}^t$, which will prevent
some $n+1=3$ data tables from being mutually joinable.

Obstructions in dimension $n-1=2$ detect spheres in $\sS{F}^t$, which will prevent
some $n+1=4$ data tables from being mutually joinable.

Obstructions in dimension $n-1=0$ detect non-path-connectedness of $\sS{F}^t$, which would
prevent some $n+1=2$ data tables from being joinable (but this is impossible with our definitions including trivial joins).
\end{rmk}

The next theorem is an adaptation of Theorem 34.6 and Corollary 34.7 in
\cite{Steenrod1943}, which is summarized in Theorem 4.5 of \cite{Lundell1960}.
It relies on defining a \emph{difference cochain} that compares a homology
class of sections.

\begin{thm}
Fix a data section $\sigma:\sS{B}_{n-1} \to \sS{S}_{n-1}$.
Suppose $\xi^t_\sigma = \delta \eta$ for some $\eta \in C^{n-1}(\sS{A};
\pi_{n-1}(\sS{F}^t))$.   Then there exists a data section
$\tau:\sS{A}_{n} \to \sS{F}^t_n$ such that $\tau|_{n-2} = \sigma|_{n-2}$. 
The converse holds as well.
\label{thm:class0}
\end{thm}

Theorem~\ref{thm:tri} restates Theorems~\ref{thm:cocycle0} and \ref{thm:class0} in practical language.
\begin{thm}[Steenrod's Trichotomy]
Fix a data subcomplex $\sS{S}$ of an ambient $\sS{X}$, with Wasserstein filtration $(\sS{F}^t)$.
Exactly one of the following is true.
\begin{enumerate}
\item $\xi^t_\sigma = e$ as a cocycle.  Every $n{-}1$-cycle of $n{+}1$ 
data tables in $\sS{S}$ over a total of $n{+}1$ attributes can be approximately joined to a single data table over those $n{+}1$ attributes, allowing error at-most $t$ in any reduction to the original data.

\item $\xi^t_\sigma \neq e$ as a cocycle, but $[\xi^t_\sigma] = e$ as a cohomology class.  
There is some $(n{-}1)$-cycle of $n{+}1$ data tables $(T_{\hat{0}},\tau_{\hat{0}}), \ldots, (T_{\hat{n}},\tau_{\hat{n}})$ in $\sS{S}$ such that the combined attribute list $T=[a_0,\ldots,a_n]$ does \emph{not} admit an approximate join $(T,\tau)$ with error at-most $t$.  However, if one considers all of the faces of these data tables, then there is an approximate join to $(T,\tau)$ of error at-most $t$.

\item $[\xi^t_\sigma] \neq e$ as a cohomology class. 
There is some $(n{-}1)$-cycle of $n{+}1$ data tables $(T_{\hat{0}},\tau_{\hat{0}}), \ldots, (T_{\hat{n}},\tau_{\hat{n}})$ in $\sS{S}$ such that the combined attribute list $T=[a_0,\ldots,a_n]$ does \emph{not} admit an approximate join $(T,\tau)$ with error at-most $t$, even when omitting attributes from the original data tables.    The only way to produce a single joined table is to increase the error threshold $t$.
\end{enumerate}
\label{thm:tri}
\end{thm}

\begin{defn}[Persistence of Obstruction]
Let $\sS{S} \subseteq \sS{F}^0 \subset \cdots \sS{F}^t \subset \cdots
\subset \sS{F}^\infty = \sS{X}$ be the filtration of a path-connected data complex.
Fix a dimension $n$ such that $d_i Y \in \sS{B}_{n-1}$ for all faces $d_i$ of all
$Y \in  \sS{A}_{n}$.
Let $\sigma:\sS{B}_{n-1} \to \sS{S}_{n-1}$ be a data section.
Fix a basepoint $(T_0,\tau_0) \in \sS{S} \subset \sS{F}^0$.
Define 
\[t_n(\sigma) \defeq \inf\{ t \,:\, \xi^t_\sigma= e \in C^n(\sS{A}; \pi_{n-1}(\sS{F}^t))\}\]
and
\[t'_n(\sigma) \defeq \inf\{ t \,:\, [\xi^t_\sigma] = e \in H^n(\sS{A}; \pi_{n-1}(\sS{F}^t))\}.\]
Note that $t'_n(\sigma) \leq t_n(\sigma)$.
\label{def:pers}
\end{defn}

\begin{rmk}
Consider a data section $\sigma:\sS{B} \to \sS{S}$. A specific value
$t_n(\sigma)=t$ means that $\sigma$ admits an extension into $\sS{F}^t$, but
not for any level of the filtration less than $t$.  In other words, there is no obstruction to extension beyond the mere existence of the data section $\sigma:\sS{B}_{n-1} \to \sS{F}^t_{n-1}$.
Similarly, by Theorem~\ref{thm:tri}, a specific value $t'_n(\sigma) = t$ means that 
there is no obstruction to extension beyond the mere existence of 
the data section $\sigma|_{n-2}:\sS{B}_{n-2} \to \sS{F}^t_{n-2}$.
\end{rmk}

\begin{rmk}
When obstructions are resolved, there are typically many
solutions to Problems \ref{GPmerge}/\ref{GPmetamerge}.  That is, if any hypothesis is consistent in \ref{GPtest},
then there are typically many other hypotheses that are consistent as well. Typical methods for choosing among them often involve
posing and then solving some optimization problem. 
We might propose enriching those optimization problems via inclusion of a measure of global inconsistency.
More precisely, the cost of a proposed data section $\sigma$ might be some combination of a local cost and some decreasing function of $t_n(\sigma)$ or $t'_n(\sigma)$;
in other words, one might penalize proposed local mergers based on the degree of difficulty they cause in forming global consensus with other local mergers.
\end{rmk}

\section{Discussion}
\label{sec:Disc}

This paper provides a mathematical foundation for semi-automated data-table-alignment tools that
are common in commercial database software. 
Data tables are abstracted as measures over value spaces, and the problem of merging tables, or indeed merging previously-merged tables, is recast as the search for a measure
that marginalizes correctly. This abstraction, and the simplicial set structure built with it, permits several advances over the current state of the art in database engineering.
Ongoing and future work will focus on developing clear algorithms for application of persistent obstruction theory to real-world database engineering and related problems in data science. 

We conclude this paper with several brief remarks about further work and also some practicalities for future use of this theory:

\begin{itemize}
\item A data sample $X$ in any metric space $V$ provides a measure, by counting.  The measure is $\mu(U) = \# (U \cap X)$ or normalized as $\mu(U) = \frac{\#(U \cap X)}{\#X}$ for any $U \in 2^V$.  

\item For computational purposes, most infinite metric spaces can be considered as compact or finite spaces, using bounds or bins or kernel methods or distributional coordinates that are appropriate to the problem at hand.

\item On the compact metric spaces $\Val(T)$, measures of interest can be described as density functions via a Radon--Nikodym comparison to the uniform probability measure $\kappa_T$.

\item One attribute can represent models on other attributes, providing an interpretation of Bayesian inference and an opportunity to apply persistent obstruction theory to compact parameterized model spaces.  In machine learning, one could use this framework to describe the compatibility of solutions in ensemble methods.

\item Any list of attributes can be considered as a single attribute, because it is still provides measures over some metric space.  There is no requirement that attribute value spaces are ``minimal'' or ``1-dimensional'' in any sense.

\item Filtrations other than $L^\infty$-Wasserstein might work, too, but someone has to prove that all levels of the filtration are fibrant.

\item To study a complex of approximate joins, $\sS{F}^t$, one must compute Wasserstein distances as in Definition~\ref{def:wasserstein}. This can be done efficiently using the tools of optimal transport as in \cite{Peyre2018}. 

\item To apply Theorem~\ref{thm:tri}, one must compute $\xi^t_\sigma$ in the simplicial homotopy group $\pi_{n-1}(\sS{F}^t)$.
This is definitely the greatest challenge for realizing these mathematical advances as actual software, because 
homotopy groups are notoriously difficult to compute in general.
The task is simplified in our case by several factors.
First, we do not necessarily need to know the group structure of $\pi_{n-1}(\sS{F}^t)$ to know whether a particular element $\xi^t_\sigma$ is trivial in that group.
Second, a data subcomplex $p:\sS{S} \to \sS{B}$ is always finitely generated with $\sS{B}$ finite, and that finite number is small (several, not several trillion) in most use-cases.
Third, because any list of attributes can be considered as a single attribute, problems that are \emph{a priori} high-dimensional can be studied with a smaller list of formal attributes.
Fourth, we expect the homotopy $\pi_{n-1}(\sS{F}^t)$ to simplify as $t$ increases, so for practical purposes it may be easy to bound $t_n(\sigma)$ even if each $\pi_{n-1}(\sS{F}^t)$ is difficult to compute.
We expect (or hope) that $\pi_1$ and $\pi_2$ are often sufficient for practical problems.

\item The most important conclusions of this work are: \emph{Any manual or
automatic data-merging system must analyze homotopy in order to guarantee success;} and \emph{Obstructions can only be resolved two ways---backing up one step, or allowing additional leeway in the data comparison.}
\end{itemize}

\newpage

\appendix
\section{Categorical Definitions}
\label{sec:sset}
This appendix provides a rapid summary of a categorical interpretation of
the development in Section~\ref{sec:db}.
For more on these topics, and for the notion of homotopy for fibrant objects in model
categories, see \cite{Kan1958, May1967, Quillen1967, Seriesa, Friedman2008}.
The reader is warned that each of these references uses a slightly different
convention for ordering, opposite categories, and co-/contra-variant functors.

\subsection{Simplex}
Let $\mathbf{Set}$ denote the set category, whose objects are sets and whose
morphisms are functions.

Let $\mathbf{\Delta}$ denote the simplex category, whose objects are the
\emph{nonempty} sets of natural numbers with the standard ordering $\leq$, written $\mathbf{n} \defeq
\{0,1,\cdots,n\}$, and whose morphisms are order-preserving functions.  
Let $\mathbf{\Delta_a}$ denote the augmented simplex category, whose objects
are sets of natural numbers with the standard ordering,
and whose morphisms are order-preserving functions.  The augmented simplicial
category includes the empty set, denoted $\mathbf{-1}$ or $\emptyset$, which is
the initial object in the category.  So,
$\mathbf{\Delta_a} = \mathbf{\Delta} \cup \{\emptyset\}$.
A monomorphism in $\mathbf{\Delta_a}$ is a one-to-one order-preserving
function.  The only bimorphisms/isomorphisms in $\mathbf{\Delta_a}$ are the
identity maps.
Among the morphisms in $\mathbf{\Delta}$ and $\mathbf{\Delta_a}$ are the co-faces $d^i$ and
co-degeneracies $s^i$, defined as follows.
\[
\begin{split}
d^i&:\mathbf{n-1} \to \mathbf{n}~\text{by}~\\
d^i&:(0,\ldots,i-1,i,i+1,\ldots,n-1) \mapsto
     (0,\ldots,i-1,i+1,i+2,\ldots,n),~\text{respectively.}\\
s^i&:\mathbf{n+1} \to \mathbf{n}~\text{by}~\\
s^i&:(0,\ldots,i-1,i,i+1,\ldots,n+1) \mapsto 
          (0,\ldots,i-1,i,i,\ldots,n),~\text{respectively.}
\end{split}
\]
These morphisms satisfy the conditions
\begin{enumerate}
\item $d^j \circ d^i = d^i \circ d^{j-1}$, if $i<j$;
\item $s^j \circ d^i = d^i \circ s^{j-1}$, if $i<j$;
\item $s^j \circ d^j = d^{j+1} \circ s^j = \mathrm{id}$;
\item $s^j \circ d^i = d^{i-1} \circ s^j$, if $i > j+1$; and 
\item $s^j \circ s^i = s^i \circ s^{j+1}$, if $i \leq j$.
\end{enumerate}
Every non-identity morphism in $\mathbf{\Delta}$ or $\mathbf{\Delta_a}$ can be written a
finite composition of co-face and co-degeneracy morphisms, so these five
properties essentially characterize $\mathbf{\Delta}$ and $\mathbf{\Delta_a}$.

For our applications, the following lemmas about monomorphisms in $\mathbf{\Delta_a}$ are very useful. 
They are elementary, but do not appear in the standard references in
this form.  Merged indexing is merely an ordered formulation of the
inclusion--exclusion principle.

\begin{lemma}[Complimentary Monomorphism]
For any monomorphism $\iota:\mathbf{n'} \to \mathbf{n}$ in $\mathbf{\Delta_a}$, 
write $m = n- n' -1$.  There is a monomorphism $\iota^c:\mathbf{m} \to
\mathbf{n}$ in $\mathbf{\Delta_a}$ that enumerates the entries of $\mathbf{n}$
that are not in the image of $\iota$.  
\label{lem:compliment}
\end{lemma}

\begin{lemma}[Merged Indexing]
In the category $\mathbf{\Delta_a}$, suppose $\mathbf{n_{0}}, \mathbf{n_{01}},
\mathbf{n_{02}}$ are 
equipped with monomorphisms 
$\iota_{01}:\mathbf{n_{0}} \into \mathbf{n_{01}}$ and 
$\iota_{02}:\mathbf{n_{0}} \into \mathbf{n_{02}}$.
Then, for $n_{012} = n_{01} + n_{02} - n_{0}$, there are monomorphisms 
$\mu_{01}:\mathbf{n_{01}} \into \mathbf{n_{012}}$ and 
$\mu_{02}:\mathbf{n_{02}} \into \mathbf{n_{012}}$ such that
$\iota_{0} \defeq \mu_{01}\circ\iota_{01} = \mu_{02}\circ
\iota_{02}:\mathbf{n_{0}}\to \mathbf{n}$ is well-defined.  Moreover, the
complementary monomorphisms $\iota_{01}^c$ and $\iota_{02}^c$ provide monomorphisms
$\iota_1$ and $\iota_2$, as in
Diagram~\eqref{eq:split}.  The images of $\iota_0, \iota_1, \iota_2$ are
disjoint.
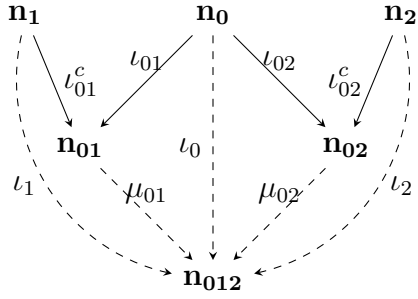
\begin{equation}
\begin{tikzpicture}[node distance=2.5cm]
\node (k0) {$\mathbf{n_{0}}$};
\node[left of=k0] (k1) {$\mathbf{n_{1}}$};
\node[right of=k0] (k2) {$\mathbf{n_{2}}$};
\node[below left of=k0] (k01) {$\mathbf{n_{01}}$};
\node[below right of=k0] (k02) {$\mathbf{n_{02}}$};
\node[below right of=k01] (k012) {$\mathbf{n_{012}}$};
\draw[->] (k0) to node[above] {$\iota_{01}$} (k01); 
\draw[->] (k0) to node[above] {$\iota_{02}$} (k02); 
\draw[->] (k1) to node[right] {$\iota_{01}^c$} (k01); 
\draw[->] (k2) to node[left] {$\iota_{02}^c$} (k02); 
\draw[dashed, ->] (k01) to node[above] {$\mu_{01}$} (k012); 
\draw[dashed, ->] (k02) to node[above] {$\mu_{02}$} (k012); 
\draw[dashed, ->] (k0) to node[left] {$\iota_{0}$} (k012); 
\draw[dashed, ->, bend right=45] (k1) to node[left] {$\iota_{1}$} (k012); 
\draw[dashed, ->, bend left=45] (k2) to node[right] {$\iota_{2}$} (k012); 
\end{tikzpicture}
\label{eq:split}
\end{equation}
\label{lem:mergeind}
\end{lemma}
\begin{proof}
The sets $\mathbf{n_{01}}, \mathbf{n_{02}}, \mathbf{n_{0}}$  
have sizes $n_0{+}1, n_{01}{+}1, n_{02}{+}1$ respectively.
Then, $n_{012} \defeq n_{01} + n_{02} - n_{0}$ satisfies $n_{012}+1 \defeq
(n_{01}+1) + (n_{02}+1) - (n_{0} + 1)$ and defines the object
$\mathbf{n_{012}}=\{0,\ldots, n_{012}\}$ in $\mathbf{\Delta_a}$.

Monomorphisms $\mu_{01}$ and $\mu_{02}$ can be constructed via the
algorithm in Figure~\ref{mergeind}, which is merely a sequence of
concatenations spliced between aligned 
entries of $\iota_{01}$ and $\iota_{02}$.    The resulting maps are indeed
morphisms, as they are guaranteed to be order-preserving.
\end{proof}

\begin{figure}[b]
\begin{lstlisting}[language=Python,frame=single]
def merge_idx(n01, n02, iota01, iota02):
    """
    Parameters: 
        n01, a non-negative integer
        n02, a non-negative integer
        iota01, an increasing list within [0,..,n01]
        iota02, an increasing list within [0,..,n02]
        (iota01 and iota02 must be the same length)
    Returns:
        mu01, an increasing list of n01+1 integers
        mu02, an increasing list of n02+1 integers
        iota0, an increasing list, same length as iota01,iota02 
    """
    n0 = len(iota01) - 1 
    n012 = n01 + n02 - n0
    mu01 = []
    mu02 = []
    i0 = i01 = i02 = i012 = 0
    while i0 <= n0:
        while i01 < iota01[i0]:
            mu01.append(i012)
            i012 += 1
            i01 += 1
        while i02 < iota02[i0]:
            mu02.append(i012)
            i012 += 1
            i02 += 1
        # now, both i01 and i02 correspond to i0
        mu01.append(i012)
        mu02.append(i012)
        i0 += 1
        i01 += 1 
        i02 += 1
        i012 += 1
    # mutual terms are extinguished. concatenate.
    while i01 <= n01:
        mu01.append(i012)
        i012 += 1
        i01 += 1
    while i02 <= n02:
        mu02.append(i012)
        i012 += 1
        i02 += 1    

    iota0 = [ mu01[i] for i in iota01 ]
    #     = [ mu02[i] for i in iota02 ]
    return ( mu01, mu02, iota0 ) 
\end{lstlisting}
\label{mergeind}
\caption{Merged indexing algorithm.}
\end{figure}

\begin{ex}
Consider $n_0=1$ and $n_{01} = 5$ and $n_{02} = 4$.  Then $n_{012} = 8$.
Let 
$\iota_{01}:\mathbf{1} \mapsto \mathbf{5}$ be the monomorphism that is
written as the sequence $[1,4]$.
Let 
$\iota_{02}:\mathbf{1} \mapsto \mathbf{5}$ be the monomorphism that is 
written as the sequence $[1,3]$.
Visually, the merged indexing means
\[
\begin{split}
\iota_{01}:\{0,1\} &\mapsto
\{0,\phantom{0,}\underline{1},2,3,\phantom{2,}\underline{4},5\phantom{,4}\}\\
\iota_{02}:\{0,1\} &\mapsto \{\phantom{0,}
0,\underline{1},\phantom{2,3,}2,\underline{3},\phantom{5,}4\}\\
&\text{yields}\\
\mu_{01}:\{0,1,2,3,4,5\} &\mapsto
\{\underline{0},1,\underline{2},\underline{3},\underline{4},5,\underline{6},\underline{7},8\}\\
\mu_{02}:\{0,1,2,3,4\}   &\mapsto
\{0,\underline{1},\underline{2},3,4,\underline{5},\underline{6},7,\underline{8}\}\\
&\text{so}\\
\iota_{0}:\{0,1\} &\mapsto \{0,1,\underline{2},3,4,5,\underline{6},7,8\}\\
\iota_{1}:\{0,1,2,3\} &\mapsto
\{\underline{0},1,2,\underline{3},\underline{4},5,6,\underline{7},8\}\\
\iota_{2}:\{0,1,2\}   &\mapsto
\{0,\underline{1},2,3,4,\underline{5},6,7,\underline{8}\}\\
\end{split}
\]
For abbreviation and programming, the constructed monomorphisms can be written as lists.
\[
\begin{split}
\mu_{01} &=  [0, 2, 3, 4, 6, 7]\\
\mu_{02} &= [1, 2, 5, 6, 8]\\
\iota_0 &= [2,6]\\
\iota_1 &= [0,3,4,7]\\
\iota_2 &= [1,5,8]\\
\end{split}
\]
\label{ex:ind}
\end{ex}

\subsection{Simplicial Sets}
For any category $\mathbf{C}$, the ``simplicial category over $\mathbf{C}$'' is
$\mathbf{sC}$. An object in $\mathbf{sC}$ is a contravariant functor 
$X:\mathbf{\Delta}\to \mathbf{C}$.  That is, an object in $\mathbf{sC}$ is an assignment of:
\begin{itemize}
\item for each object $\mathbf{n}$ in $\mathbf{\Delta}$, an object $X_n$ in $\mathbf{C}$;
\item for each morphism (order-preserving function) $\mu:\mathbf{n'} \to \mathbf{n}$ in
$\mathbf{\Delta}$, a morphism $X(\mu):X_{n} \to X_{n'}$ in $\mathbf{C}$.
\end{itemize}
The augmented simplicial category, $\mathbf{asC}$, allows 
a terminal object in $\mathbf{C}$ to correspond to the initial
object $\mathbf{-1} \in \mathbf{\Delta_a}$. That is, the trivial
map $\mathbf{-1}\to \mathbf{n}$ yields a corresponding map $X_n \to X_{-1}$, if
the category $\mathbf{C}$ happens to admit a terminal object.

The morphisms $X \to Y$ in $\mathbf{sC}$ or $\mathbf{asC}$ are the natural
transformations as in \eqref{eqn:natsS}.
\begin{equation}
\begin{tikzpicture}[node distance=2cm]
\node (X) {$X$};
\node[below of=X] (Y) {$Y$};
\node[right of=X] (Xnp) {$X_{n'}$};
\node[right of=Y] (Ynp) {$Y_{n'}$};
\node[right of=Xnp] (Xn) {$X_n$};
\node[right of=Ynp] (Yn) {$Y_n$};
\node (np) [below =0.5cm of Ynp] {$\mathbf{n'}$};
\node (n) [below =0.5cm of Yn] {$\mathbf{n}$};
\draw[->] (X) to node[right] {} (Y);
\draw[->] (Xn) to node[above] {$X(\mu)$} (Xnp);
\draw[->] (Yn) to node[above] {$Y(\mu)$} (Ynp);
\draw[->] (np) to node[above] {$\mu$} (n);
\draw[->] (Xnp) to node[left] {} (Ynp);
\draw[->] (Xn) to node[left] {} (Yn);
\end{tikzpicture}
\label{eqn:natsS}\end{equation}

The most important case is $\mathbf{sSet}$, the category of simplicial sets, which is
augmented to $\mathbf{asSet}$.  The following lemma shows that augmented
simplicial sets are given by face and degeneracy maps.
\begin{lemma}
Any object in $\mathbf{asSet}$ is a set $X$ (called an \emph{augmented
simplicial set}) graded by $-1,0,1,2,\ldots$ and
equipped with morphisms $d_i:X_{n} \to X_{n-1}$ and $s_i:X_{n}
\to X_{n+1}$ for $0\leq i \leq n$ such that
\begin{enumerate}
\item $d_i \circ d_j = d_{j-1} \circ d_i$, if $i<j$;
\item $d_i \circ s_j = s_{j-1}\circ d_i$, if $i<j$;
\item $d_j \circ s_j = d_{j+1} \circ s_j= \mathrm{id}$;
\item $d_i \circ s_j = s_j \circ d_{i-1}$, if $i > j+1$; and 
\item $s_i \circ s_j = s_{j+1}\circ s_i$, if $i \leq j$.
\end{enumerate}
\label{lem:generate}
\end{lemma}
\begin{proof}
The objects are apparent.  As for morphisms, each co-face 
$d^i:\mathbf{n-1}\to\mathbf{n}$ and co-degeneracy $s^i:\mathbf{n+1} \to
\mathbf{n}$ morphism in $\mathbf{\Delta_a}$ must correspond to 
face $d_i:X_n \to X_{n-1}$ and boundary $s_i:X_n \to X_{n+1}$ morphisms in $X$.
Because the co-face and co-degeneracy morphisms generate all non-identity morphisms in
$\mathbf{\Delta_a}$, it is sufficient to specify these face an degeneracy maps.
\end{proof}

\begin{cor}[Simplicial Maps]
The morphisms of $\mathbf{sSet}$ or $\mathbf{asSet}$ (called \emph{simplicial
maps}) from \eqref{eqn:natsS} are set functions $f:X \to Y$ such that $d_i \circ f = f \circ d_i$ and $s_i \circ f = f \circ s_i$.
\end{cor}

A particularly important example of a simplicial set is $\Delta^n$, the
$n$-simplex. (See \ref{sec:DJKC}.)
\begin{defn}[Simplex]
The standard $n$-simplex $\Delta^n$ is the simplicial set generated (via face and
degeneracy maps) by the ordered set $\mathbf{n}=\{0,\ldots,n\}$ in the simplex
category $\mathbf{\Delta}$.
\label{def:simplex}
\end{defn}
By the Yoneda Lemma, a simplicial set $X$ is characterized by the simplicial
maps $\Delta^n \to X$; that is, a simplicial set is characterized by its
simplices.  

\begin{defn}[Horn]
The $k$th \emph{horn} $\Lambda^n_k$ of the $n$-simplex $\Delta^n$ is the
simplicial subset generated by the union of all the faces of $\Delta^n$ except the $k$th face.
\label{def:horn}
\end{defn}
By Lemma~\ref{lem:generate} and the Yoneda Lemma, if $X$ is a simplicial set,
then a horn in $X$ is a collection of $n$ $(n{-}1)$-simplices $f_0, \ldots,
f_{k-1}, f_{k+1}, \ldots, f_n$ such that $d_i f_j = d_{j-1}f_i$ for $i<j$.

A simplicial map $f:X \to Y$ is called a \emph{cofibration} iff it is a
monomorphism.  A simplicial map $f:X \to Y$ is called a \emph{fibration} 
iff for any cofibration $i:\Lambda^n_k \into \Delta^n$, the
commutative diagram \eqref{eqn:fibSet} can be completed.
\begin{equation}
\begin{tikzpicture}[node distance=2cm]
\node (X) {$X$};
\node[below of=X] (Y) {$Y$};
\node[left of=X] (H) {$\Lambda^n_k$};
\node[left of=Y] (D) {$\Delta^n$};
\draw[->] (X) to node[right] {$f$} (Y);
\draw[->] (H) to node[left] {$i$} (D);
\draw[->] (H) to node[above] { } (X);
\draw[->] (D) to node[above] { } (Y);
\draw[->,dashed] (D) to node[above] { } (X);
\end{tikzpicture}
\label{eqn:fibSet}\end{equation}
Weak-equivalences are defined to be compatible with fibrations and cofibrations
according to \cite{Quillen1967}.  See also \cite{Seriesa}.  These definitions
of cofibration, fibration, and weak equivalence make $\mathbf{sSet}$ into a
(closed) model category.

A simplicial set $X$ is called \emph{fibrant} or to satisfy the \emph{Kan
extension condition} if $f:X \to \{*\}$ is a fibration; that is, a simplicial
set satisfies the Kan condition if and only if each horn $\Lambda^n_k$ in $X$
can be extended to a simplex $\Delta^n$ in $X$.  Let $\mathbf{sSet}_f$
denote the subcategory of fibrant simplicial sets.  Then there is a homotopy
category $\Pi_n(\mathbf{sSet}_f)$, and any $X \in \mathbf{sSet}_f$ admits
pointed homotopy groups $\pi_n(X,x)$ that characterize the weak equivalence.
Moreover, the simplicial homotopy groups of $X \in \mathbf{sSet}_f$ are
isomorphic to the continuous homotopy groups of its topological realization,
$|X|$,  as discussed in \cite[\S3]{Quillen1967} and
\cite[Chap~I.2]{Seriesa}.  See also \cite{Kan1958} and \cite{May1967} for
historical explanations that minimize categorical language.

\subsection{Data Complexes}
Let $\mathbf{DataCplx}$ denote the category of data complexes.
An object in $\mathbf{DataCplx}$ is a pair of augmented simplicial sets
$(\sS{X},\sS{A})$ with simplicial map $p:\sS{X} \to \sS{A}$ such that
for each $\mathbf{n} \in \Delta_a$, the set $\sS{X}_n$ is a set of data tables
over attribute lists $\sS{A}_n$ from some attribute set $A$, as in Section~\ref{sec:db}, with $d_i$ and $s_i$
by marginalization and Dirac-delta intersection, respectively.

A morphism in $\mathbf{DataCplx}$ is simplicial map $f:(\sS{X},\sS{A}) \to
(\sS{Y},\sS{B})$ as in \eqref{eqn:natDT} with some compatibility conditions.
\begin{equation}
\begin{tikzpicture}[node distance=2cm]
\node (X) {$\sS{X}$};
\node[below of=X] (Y) {$\sS{Y}$};
\node[right of=X] (Xnp) {$\sS{X}_{n'}$};
\node[right of=Y] (Ynp) {$\sS{Y}_{n'}$};
\node[right of=Xnp] (Xn) {$\sS{X}_n$};
\node[right of=Ynp] (Yn) {$\sS{Y}_n$};
\node (np) [below =0.5cm of Ynp] {$\mathbf{n'}$};
\node (n) [below =0.5cm of Yn] {$\mathbf{n}$};
\draw[->] (X) to node[right] {$f$} (Y);
\draw[->] (Xn) to node[above] {$\sS{X}(\mu)$} (Xnp);
\draw[->] (Yn) to node[above] {$\sS{Y}(\mu)$} (Ynp);
\draw[->] (np) to node[above] {$\mu$} (n);
\draw[->] (Xnp) to node[left] {} (Ynp);
\draw[->] (Xn) to node[left] {} (Yn);
\end{tikzpicture}
\label{eqn:natDT}
\end{equation}
The vertical maps are tuples $(\varphi_n, \{ \psi_a \}_{a \in A}, f_n)$
satisfying the following compatibility conditions.
\begin{enumerate}
\item $\varphi_n:\sS{A}_n \to \sS{B}_n$ is a level of a simplicial map
$\varphi:\sS{A} \to \sS{B}$ on sets of attribute lists.
\item $f_n:\sS{X}_n \to \sS{Y}_n$ is a level of a simplicial map
$f:\sS{X} \to \sS{Y}$ on sets of measures, with $\varphi_n = p \circ f_n$.
\item $\psi_a:\Val([a]) \to \Val([b])$ is a continuous function on
metric spaces, where $[b] = \varphi_{0}([a]) \in \sS{B}_{0}$.
This induces $\psi_T:\Val(T) \to \Val(\varphi_{n}(T))$ for all $T
\in \sS{A}_n$.
\item If $(T,\tau) \in \sS{X}_n$ with $\varphi_n(T) = S$ and $f_n(T,\tau)
= (S,\sigma)$, then $\tau \circ \psi_T^{-1} = \sigma$ as measures.
That is,
\begin{equation}
f_n:(T,\tau) \mapsto (\varphi_n(T), \tau \circ \psi_T^{-1}).
\label{eqn:mapcompat}
\end{equation}
\end{enumerate}
These conditions guarantee simply that the attribute lists $T$, the value
spaces $\Val(T)$, and the measure spaces $\Msr(T)$ remain compatible.
As with $\mathbf{sSet}$, in \eqref{eqn:natDT}, the map $\mu:\mathbf{n'} \to
\mathbf{n}$ can be taken to be $d^i:\mathbf{n-1} \to \mathbf{n}$ or
$s^i:\mathbf{n+1} \to \mathbf{n}$ so that the diagram describes naturality with
respect to face and degeneracy maps on $\sS{X}$ and $\sS{Y}$.
These conditions are sensible for $n\geq 0$, so they apply to the trivial data
table $([ ], M)$. 

For each real number $M \geq 0$, there is a \emph{singleton} data complex with
$A = \{ * \}$, $\Val(*) = \{ *\}$.  For each $n \geq -1$, there is a single
attribute list $[*, \ldots, *]$ with a singleton value space  $\{*\}^n$ and one
measure, $M$.   For brevity, we refer to this singleton data complex as $M$.

Slightly more generally, there is a \emph{terminal} data complex with
$A = \{ * \}$, $\Val(*) = \{ * \}$.  For each $n \geq -1$, there is a single
attribute list $[*, \ldots, *]$ with a singleton value space  $\{*\}^n$ and
measures $M$ for each $M \geq 0$.   The terminal data complex is the union of
all the singleton data complexes.  For brevity, we refer to the terminal data
complex as $\mathbb{R}_{\geq 0}$.

Every data complex $\sS{X}$ admits a morphism to the terminal data complex $\mathbb{R}_{\geq 0}$.  This
terminal morphism $f$ maps each data table $(T, \tau) \in \sS{X}$ to the
singleton mass $([*, \ldots, *], \tot \tau) \in \mathbb{R}_{\geq 0}$.
If all data tables in $\sS{X}$ share the same mass (say, $M=1$), then the image
of the terminal morphism goes to some $M \subset \mathbb{R}_{\geq 0}$. 

A morphism in $\mathbf{DataCplx}$ is called a cofibration iff it is a
monomorphism.  A morphism in $\mathbf{DataCplx}$ is called a fibration iff 
for any cofibration of from a well-aligned pair to a join
$i:\pair{\tau_{01},\tau_{02}}_{T_0} \to \pair{\tau_{012}}$, the
commutative diagram \eqref{eqn:fibData} can be completed.
\begin{equation}
\begin{tikzpicture}[node distance=2cm]
\node (X) {$\sS{X}$};
\node[below of=X] (Y) {$\sS{Y}$};
\node[left of=X] (H) {$\pair{\tau_{01},\tau_{02}}_{T_0}$};
\node[left of=Y] (D) {$\pair{\tau_{012}}$};
\draw[->] (X) to node[right] {$f$} (Y);
\draw[->] (H) to node[left] {$i$} (D);
\draw[->] (H) to node[above] { } (X);
\draw[->] (D) to node[above] { } (Y);
\draw[->,dashed] (D) to node[above] { } (X);
\end{tikzpicture}
\label{eqn:fibData}
\end{equation}

A data complex $\sS{X}$ is called \emph{fibrant} if the terminal morphism $\sS{X} \to
\mathbb{R}_{\geq 0}$
is a fibration.
By Theorem~\ref{thm:ft}, if $\sS{X}$ is a fibrant data complex, then $\sS{X}$ is a fibrant simplicial set.
Thus, the category  $\mathbf{DataCplx}$ is a (closed) model
category, and the fibrant subcategory $\mathbf{DataCplx}_f$ inherits a 
well-defined homotopy category $\Pi_n(\mathbf{DataCpl}_f)$ from
$\mathbf{sSet}_f$,
and any $\sS{X} \in \mathbf{DataCplx}_f$ admits
pointed homotopy groups $\pi_n(\sS{X},\tau_0)$ that characterize the weak equivalence.
Moreover, the homotopy groups 
are isomorphic to the continuous homotopy groups of the topological realization
of the underlying simplicial set.

\newpage
\providecommand{\bysame}{\leavevmode\hbox to3em{\hrulefill}\thinspace}
\providecommand{\MR}{\relax\ifhmode\unskip\space\fi MR }
\providecommand{\MRhref}[2]{%
  \href{http://www.ams.org/mathscinet-getitem?mr=#1}{#2}
}
\providecommand{\href}[2]{#2}

\end{document}